\documentclass[english,11pt]{amsart}

\usepackage[english]{babel}
\usepackage[T1]{fontenc}
\usepackage{lmodern}

\usepackage{MnSymbol}
\usepackage[centering]{geometry}

\theoremstyle{plain}
\newtheorem{theo}{Theorem}
\newtheorem{lemm}[theo]{Lemma}

\newtheorem{prop}[theo]{Proposition}
\newtheorem{coro}[theo]{Corollary}

\theoremstyle{definition}

\theoremstyle{remark}
\newtheorem{rema}[theo]{Remark}

\usepackage{microtype}
\usepackage{enumerate}
\usepackage{graphicx}
\usepackage{color}
\usepackage{bm}
\usepackage{mathtools}
\usepackage[dvipsnames]{xcolor}
\definecolor{FlatRed}{RGB}{231,76,60}
\definecolor{FlatGreen}{RGB}{46,204,113}
\definecolor{FlatBlue}{RGB}{52,152,219}
\definecolor{FlatYellow}{RGB}{241,196,15}
\colorlet{FlatViolet}{FlatRed!50!FlatBlue}
\colorlet{FlatBrown}{FlatRed!50!FlatGreen}
\colorlet{FlatOrange}{FlatRed!50!FlatYellow}
\colorlet{FlatCyan}{FlatGreen!50!FlatBlue}

\usepackage{enumitem}
\usepackage{textcomp}
\usepackage[normalem]{ulem}

\usepackage{hyperref}
\hypersetup{
    colorlinks,
    linkcolor={FlatRed},
    citecolor={FlatGreen},
    urlcolor={FlatBlue}
}

\title[The extended Seiberg bound in the boundary case]{Moment bounds for Gaussian multiplicative chaos with higher-dimensional singularities}

\author{Yichao Huang}
\address{Beijing Institute of Technology, School of Mathematics and Statistics, Beijing, China}
\email{yichao.huang@bit.edu.cn}

\begin{document}

\begin{abstract}
We determine the exact threshold of the extended Seiberg bound for the existence of correlation functions in the boundary Liouville conformal field theory in the unit disk. In probabilistic terms, our result is a toolbox yielding the threshold for the existence of positive moments of Gaussian multiplicative chaos measure, appliable to the case where singularities of arbitrary (co-)dimension in the background metric are present. We improve previous results of this type for $0$-dimensional singularities in~\cite{david2016liouville} and a sufficient condition for the $1$-dimensional singularity in an unpublished appendix of~\cite{huang2018liouville}. In particular, we prove the optimality of the moment bound threshold  for boundary Gaussian multiplicative chaos conjectured in~\cite{huang2018liouville}, which is equivalent to the so-called unit volume Seiberg bound of the boundary Liouville conformal field theory.
\end{abstract}

\maketitle

\section{Introduction}

\subsection{Motivations and backgrounds}

Gaussian multiplicative chaos measures are a natural family of multifractal random measures. Their study is motivated by the works of Kolmogorov~\cite{kolmogorov1941local}, Kolmogorov and Obukhov~\cite{kolmogorov_1962}, Mandelbrot~\cite{mandelbrot1972possible} in modeling intermittent phenomena in turbulence among many others, and Kahane~\cite{kahane1985chaos} formulated a mathematical theory of Gaussian multiplicative chaos following a suggestion of L\'evy. The study of positive moment bounds in the theory of Gaussian multiplicative chaos started already with the seminal paper of Kahane~\cite{kahane1985chaos}, and finds its root in the pre-sequel on Gaussian multiplicative cascades, e.g. in the work of Kahane and Peyri\`ere~\cite{kahane1976certaines}. In the latter, a beautiful elementary argument is devised to give precise moment bounds for $p\geq 1$ for classical multiplicative cascade models. Since the classical multiplicative cascade models (as well as the classical Gaussian multiplicative chaos models) all have finite first moment, the argument of~\cite{kahane1976certaines} is based on induction over $p\in]k,k+1]$ for integer $k\geq 1$, with meticulous manipulations of basic inequalities.

Moments of Gaussian multiplicative chaos are important in many applications of the theory, especially to problems stemmed from mathematical physics. During the recent developments in the mathematical study of Liouville conformal field theory proposed by Polyakov~\cite{POLYAKOV1981207} and its generalizations, understanding the behavior of different moments of Gaussian multiplicative chaos measures, with or without singularities (e.g. the so-called marked points or vertex operators for $0$-dimensional singularities), is of fundamental importance. This connection was first discovered by David, Kupiainen, Rhodes and Vargas in~\cite{david2016liouville}, where they showed that the $N$-point correlation functions of Liouville conformal field theory on the Riemann sphere can be expressed in terms of moments of Gaussian multiplicative chaos measure of the full-plane Gaussian free field, chaos measure which is integrated against pointwise log-singularities coming from the insertion points. This kind of probabilistic expressions via Gaussian multiplicative chaos extends to many other aspects in conformal field theory and related domains, and a far from complete list of results in this direction includes~\cite{david2016liouville,Guillarmou2019,Kupiainen_2020,guillarmou2020conformal,guillarmou2021segal,huang2018liouville,Remy_2020,remy2022integrability,cercle2021unit,aru2017two,ang2021fzz,ghosal2020probabilistic,ang2022moduli,cercle2021probabilistic,cercle2022ward,lacoin2022path,garban2020dynamical,Rhodes_2019,wong2020universal,kupiainen2019stress} and many many more.

In the current article, we are concerned with the basic question of the existence of positive moments for Gaussian multiplicative chaos measures with higher dimensional singularities. It is known since~\cite{huang2018liouville} that this problem is equivalent to the so-called extended Seiberg bounds of Liouville conformal field theory on geometries with boundary, that we shall briefly describe.

In the mathematical study of Liouville conformal field theory with boundary~\cite{huang2018liouville}, Rhodes and Vargas together with the author encountered naturally a new type of Gaussian multiplicative chaos, which is intimated related to hyperbolic geometry. Roughly speaking, in the study of Liouville conformal field theory with boundary, one would expect to study basic properties of the Gaussian multiplicative chaos measure integrated against a hyperbolic background metric (raised to certain power determined by the parameter of the Liouville conformal field theory). An important example is the case of the unit disk, for which the background metric is given by a positive power of the Poincar\'e metric, which exhibits strong divergence along the boundary, here the unit circle. The classical singularities originated from the so-called marked-points (or vertex operators) are now replaced by a uniform higher-dimensional singularity, e.g. singularity coming uniformly from the blowup of the background metric near the boundary unit circle in the case of the unit disk. A natural variant of Gaussian multiplicative chaos measures appears, with surprisingly albeit naturally nice properties, with the remarkable distinction from the classical case that this new family of random measures can fail to possess finite first moment. This lack of larger than first moment is a manifest distinctive trait of the boundary Gaussian multiplicative chaos, and renders the analysis thereof difficult.

\subsection{Seiberg bounds and probabilistic methods}

We give a brief review about the Seiberg bounds and the extended (i.e. probabilistic) Seiberg bounds.

Seiberg bounds are first proposed by Seiberg in his pioneering review~\cite{seiberg1990notes} for some form of the so-called $2$-point Liouville correlation functions. In general, Seiberg bounds form a set of necessary and sufficient conditions to ensure existence of Liouville correlation functions on any topological surface with arbitrarily many marked points, called vertex operators (with real parameters) in the physics literature. This should be compared with the classical analogue of conical singularities~\cite{Troyanov1991PrescribingCO}. In the probabilistic construction of David, Kupiainen, Rhodes and Vargas~\cite{david2016liouville}, this prediction was rigorously proven with Gaussian multiplicative chaos techniques.

In plain words, Seiberg bounds, in the form that we are interested in here, come in two parts. On the one hand we have a ``local'' condition, which states that each parameter of marked points cannot be too large to ensure integrability around each insertion point, and on the other we have a ``global'' condition, which states that the sum of parameters of marked points should be large enough to satisfy a Gauss-Bonnet flavored condition in order to have the finiteness of the total volume of the whole space.

In particular, in~\cite{david2016liouville}, with the probabilistic interpretation of Liouville correlation functions as moments of Gaussian multiplicative chaos, one sees the Seiberg bounds in a very precise way. For clarity, let us recall the expression for the Liouville correlation functions on the Riemann sphere without diving too much into details. Following~\cite{david2016liouville}, the $N$-point Liouville correlation functions on the Riemann sphere, with marked points $\{z_k\}_{1\leq k\leq N}\in\mathbb{C}^{N}$ of respective parameters $\{\alpha_{k}\}_{1\leq k\leq N}\in\mathbb{R}^{N}$, is expressed in terms of a Gaussian multiplicative chaos measure $\mu^{\gamma}$ with parameter $\gamma\in(0,2)$ via the following formula (we take the cosmological constant equal to $1$ for simplicity):
\begin{equation*}
    \left\langle \prod_{k=1}^{N}V_{\alpha_k}(z_k)\right\rangle_{\gamma}\coloneqq 2\gamma^{-1}\Gamma(s)\prod_{i<j}|z_i-z_j|^{-\alpha_i\alpha_j}\mathbb{E}\left[\left(\int_{\mathbb{C}}F(x,\bm{z})\mu^{\gamma}(d^2x)\right)^{-s}\right]
\end{equation*}
where $Q=\frac{2}{\gamma}+\frac{\gamma}{2}$ and $s=\frac{\sum_{k=1}^{N}\alpha_k-2Q}{\gamma}$, as long as the expression on the right hand side makes sense. Here the function $F$ can be view as a integral kernel with point-wise singularities whose strengths are dictated by the parameters $\{\alpha_k\}_{1\leq k\leq N}$:
\begin{equation*}
    F(x,\bm{z})\coloneqq\prod_{k=1}^{N}\left(\frac{|x|\vee 1}{|x-z_k|}\right)^{\gamma\alpha_k}.
\end{equation*}
See \cite[Section~2.2]{vargas2017lecture} for a quick introduction to the notations.

One observes two possible obstructions to the existence of Liouville correlation functions in this case. The first one, purely algebraic in the above formula, is the presence of a special function, namely $\Gamma(s)$. This corresponds exactly to the prediction of Seiberg, and yields one of the original Seiberg bounds, i.e. the ``global'' condition $\sum_{k}\alpha_k>2Q$. The other obstruction, which is of probabilistic nature, comes from the moment term of Gaussian multiplicative chaos measure. Especially, one has to check that the an insertion of large parameter will not create a singularity too important in the function $F$ for the Gaussian multiplicative chaos to absorb locally.

The condition for the second obstruction above is called the extended Seiberg bound, or probabilistic Seiberg bound, according to the terminology of~\cite{david2016liouville}. One can see it as a quantum analogue of integrability of singular kernels. As we shall remind in detail shortly, the Gaussian multiplicative chaos is realized as a multifractal measure, thus it should in principle allow high order singular kernels compared to regular integral with continuous functions. This for the reason that, at least intuitively, a fractal type of random measure will only capture parts of the singular kernel, therefore integrates to a ``thinned'' version of the singular kernel.

The extended Seiberg bound for discete marked points, or $0$-dimension singularities, was settled in~\cite{david2016liouville}. In the study of Liouville conformal field theory with boundary via Gaussian multiplicative chaos methods, Rhodes, Vargas and the current author encountered a new type of Gaussian multiplicative chaos measure, taking into account the boundary influence to the bulk measure. That is, we realized in the course of study that the boundary effect can be viewed as uniform $1$-dimension singularity with certain parameter according to the path integral formulation. We describe now our main results with this setting, and details about this new type of Gaussian multiplicative measure will be reviewed in~Section~\ref{sec:Background}.\footnote{We should also point out that the thorough investigation of the integrability (in the sense of random singular kernels) of this kind of higher-dimensional singularities by the Gaussian multiplicative chaos measure was also raised as a question by an anonymous referee of~\cite{huang2018liouville}. The current paper gives a complete answer this excellent question, by developing a general cut-off method robust enough to treat this type of questions for singularities with arbitrary (co)-dimension and boundaries with general symmetry.}

\subsection{Main results}
We call a centered Gaussian field $X$ boundary log-correlated (with respect to the boundary $\mathbb{R}=\partial\mathbb{H}$ unless otherwise mentioned), if it has the following Neumann-type covariance kernel in some bounded open domain $\Omega\subset\mathbb{C}$:
\begin{equation*}
    \forall z,w\in\mathbb{C},\quad \mathbb{E}[X(z)X(w)]\eqqcolon K_{\text{N}}(z,w)=-\ln|z-w||z-\overline{w}|+g(z,w)
\end{equation*}
with a bounded correction term $g(z,w)$. Without loss of generality, we will always suppose that $\Omega$ contains the origin, i.e. $0\in\Omega$, and that $\Omega$ is small enough so that $K_{\text{N}}(z,w)$ is positive and well-defined as a covariance kernel in $\Omega$.

The object of interest is the following Gaussian multiplicative chaos measure associated to $X$, reweighted by the hyperbolic metric on $\mathbb{H}$. We consider, for all compact Borel sets $A\subset\Omega$, the hyperbolic Gaussian multiplicative chaos measure
\begin{equation}\label{eq:DefinitionBoundaryGMC}
    \mu^{\gamma}_{\text{H}}(A)\coloneqq\int_{A}\text{Im}(z)^{-\frac{\gamma^2}{2}}\mu^{\gamma}(dz),
\end{equation}
where $\mu^{\gamma}(dz)$ is the classical (i.e. un-reweighted) Gaussian multiplicative chaos measure associated to the log-correlated field $X$ in the interior of $\mathbb{H}$, see Section~\ref{sec:Background} for a reminder of its definition. The random measure $\mu^{\gamma}_{\text{H}}(A)$ appears naturally in the study of boundary Liouville conformal field theory, and its moments (with possibly singularities coming from bulk or boundary marked points) are closely related to the boundary Liouville correlation functions, see Remark~\ref{rema:ShortPhysics}. It is shown in~\cite[Section~3]{huang2018liouville} that for any fixed Carleson cube $Q$ (see below for a reminder), the measure $\mu^{\gamma}_{\text{H}}(A)$ is non-degenerate if and only if $\gamma<2$, and in~\cite[Section~4]{huang2018liouville} for the critical parameter $\gamma=2$, suitable renormalizations give rise to the so-called critical Gaussian multiplicative chaos measure (in the boundary case). We shall treat the critical case $\gamma=2$ separately in a following work.

Denote by $\mathcal{Q}^{*}_{\mathbb{R}}$ the collection of Carleson cubes constructed from a compact interval of $\mathbb{R}$, i.e.
\begin{equation}
    \mathcal{Q}^{*}_{\mathbb{R}}=\left\{Q_{[a,b]}=[a,b]\times[0,|b-a|]\subset\mathbb{H}~;~-\infty<a<b<\infty\right\}.
\end{equation}

Our main result is:
\begin{theo}[The extended Seiberg bound for the boundary Liouville conformal field theory]\label{th:Main1}
    Let $X$ be a boundary log-correlated Gaussian field defined on some bounded open domain $\Omega\subset\mathbb{C}$ containing the origin. Let $\gamma\in(0,2)$ and let $\mu^{\gamma}_{\text{H}}$ be the hyperbolic Gaussian multiplicative chaos measure as in~\eqref{eq:DefinitionBoundaryGMC}. Let $Q\in\mathcal{Q}^{*}_{\mathbb{R}}$ be a Carleson cube near the origin and contained in $\Omega$. Then the following moment
    \begin{equation*}
        \mathbb{E}\left[\mu^{\gamma}_{\text{H}}(Q)^{p}\right]
    \end{equation*}
    exists in $\mathbb{R}_{>0}$ if and only if $p<\frac{2}{\gamma^2}$.
\end{theo}
To avoid confusions from mixing terminologies in probability and in physics, we focus on this mathematically flavored formulation for the most part of the article (but see Remark~\ref{rema:ShortPhysics} for some backgrounds and relations to physics).

\begin{rema}\label{rem:CoveringDisk}
Although we restrict the study of moments to compact Carleson cubes, this result is sufficient to cover the case where the whole boundary of some boundary Liouville conformal field theory is considered. Indeed, an equivalent description of boundary Liouville conformal field theory with compact boundaries is based on the disk model $\mathbb{D}\subset\mathbb{C}$, with compact boundary the unit circle $\mathbb{T}=\partial\mathbb{D}$. By a standard conformal transformation (e.g. Cayley transform), the image of an above-defined Carleson cube covers a positive portion of the unit circle. By compactness of $\mathbb{T}=\partial\mathbb{D}$, we can cover the whole unit circle by a finite number of images of such cubes. Therefore Theorem~\ref{th:Main1} is sufficient to treat the case with the whole boundary involved, idem for Liouville theory defined on manifolds with multiple compact boundaries.
\end{rema}

\begin{rema}
For the most part of this paper, we will be concerned with the case $\gamma\in(0,2)$. The parameter $\gamma$ will be dropped in most of the notations when there is no confusion. The critical case $\gamma=\gamma_c=2$ requires several addtional technicalities, and will be treated in a separate upcoming work.
\end{rema}

\subsection*{Acknowledgement} We warmly thank R\'emi Rhodes and Vincent Vargas for many discussions on boundary Gaussian multiplicative chaos measures. This work is partially supported by National Key R\&D Program of China (No. 2022YFA1006300). Support from ERC Advanced Grant 741487 QFPROBA and hospitality of the University of Helsinki are also gratefully acknowledged.

\section{Mathematical background}\label{sec:Background}
In the rest of this article, all constants $C$ may change from line to line, and we keep their dependencies in the index when necessary (e.g. $C_p$ means a constant depending on $p$).

\subsection{Classical Gaussian multiplicative chaos measures}
The classical Gaussian multiplicative chaos measure is commonly described in plain words as the ``exponential of a log-correlated Gaussian field''. The basic construction is to take a log-correlated Gaussian field $X$ on some domain $\Omega\subset\mathbb{R}^{d}$ in arbitrary dimension $d\geq 1$, i.e. $X$ is a centered Gaussian field with covariance kernel (where $g$ is a bounded function on $\Omega$)\footnote{In many applications, the correction term $g$ comes with extra regularity assumptions, such as $g\in\mathcal{C}^{0}$ or $g\in W^{s,2}$ for some $s>1$. For the main result of this paper, we only need the boundedness of $g$.}
\begin{equation}\label{eq:LogCovarianceKernel}
    K(z,w)\coloneqq\mathbb{E}[X(z)X(w)]=-\ln|z-w|+g(z,w),
\end{equation}
and define the (properly renormalized) exponential $:e^{\gamma X(z)}:$ as a random measure on Borel sets of $\Omega$. The main issue with the direct exponential of the field $X$ is the explosion of the kernel $K(z,w)$, namely that on the diagonal
\begin{equation*}
    \lim_{|z-w|\to 0}K(z,w)\to\infty,
\end{equation*}
meaning that $X$ can only be seen as a random generalized function (i.e. distribution in the sense of Schwartz). A natural workaround via regularization of the field $X$ and renormalization of the exponential was developed since Kahane~\cite{kahane1985chaos}.

There are nowadays several equivalent constructions and definitions of the classical Gaussian multiplicative chaos measure, e.g.~\cite{robert2010gaussian,SHAMOV20163224,Berestycki_2017,junnila2019decompositions}. A convenient construction for the purpose of this paper is via mollifying the Gaussian field $X$ and taking the limit of the renormalized exponentials of the regularized fields. We briefly summarize the procedure below and refer the readers to e.g.~\cite{Rhodes_2014} for additional background and details.

Let $X_{\epsilon}$ be a smooth Gaussian field approximation of the Gaussian field $X$. In this paper, it is enough to use the so-called circle-average approximation, where $X_{\epsilon}(z)$ is the average of $X$ over the Euclidean circle of center $z$ and radius $\epsilon$. Since the rest of this paper has a strong hyperbolic flavor, one can switch later to hyperbolic circles if we want to respect the hyperbolic nature, but it does not make any difference, since it is now a general result~\cite{kahane1985chaos,robert2010gaussian,SHAMOV20163224,Berestycki_2017,junnila2019decompositions} that the following limit, in the sense of weak convergence of measures, exists in probability and is independent of the regularization (as long as some mild regularity properties of the regularization are satisfied):
\begin{equation*}
    \lim_{\epsilon\to 0}\mu^{\gamma}_{\epsilon}(A)\coloneqq\lim_{\epsilon\to 0}\int_{A}e^{\gamma X_{\epsilon}(z)-\frac{\gamma^2}{2}\mathbb{E}[X_\epsilon(z)^2]}\sigma(dx)
\end{equation*}
for any measurable set $A\subset\Omega$ and any Radon measure $\sigma$. The limit measure, which we will denote by $\mu^{\gamma}$, is known as the Gaussian multiplicative chaos measure of the Gaussian field with parameter $\gamma>0$, associated with the log-correlated field $X$ (and with respect to the background measure $\sigma$). Kahane and many other authors later have shown that the limit measure is non-degenerate if and only if $\gamma\in(0,\sqrt{2d})$, and while this is the case, it has positive moment of order $p>0$ if and only if $p<\frac{2d}{\gamma^2}$. For more histories, backgrounds, proofs and applications, we recommend~\cite{Rhodes_2014,Berestycki_2017}.

We also note that the critical case $\gamma=\sqrt{2d}$ can be treated with refined renormalization procedures, all yielding the same limit measure, which is refered to as the critical Gaussian multiplicative chaos: since we will only consider this case in a separate article, we leave the curious reader to the recent review~\cite{powell2020critical}.

We sometimes refer to the parameter $\gamma$ as the ``coupling constant'' following terminologies from the physics literature.

\subsection{Gaussian multiplicative chaos measure reweighted by the hyperbolic metric}
In the study of Liouville conformal field theory with boundary, the suitable background metric for the relevant Gaussian multiplicative chaos measure is the natural hyperbolic metric (raised to a suitable power). We refer to~\cite[Section~2.4]{huang2018liouville} for a brief review of the probabilistic construction of the boundary Gaussian multiplicative chaos, and only gather the relevant information here.

More precisely, in the boundary Liouville conformal field theory, the relevant Gaussian multiplicative chaos is the one associated with the Gaussian free field with Neumann boundary conditions. In the half-plane model denoted by the index $\text{H}$ or $\mathbb{H}$, consider the Gaussian multiplicative chaos measure $\mu^{\gamma}$ associated with the Neumann kernel, the latter being a log-correlated kernel of the following form defined on some domain $\Omega\subset\mathbb{H}$, that we should always suppose without loss of generality contains the origin, i.e. $0\in\Omega$:
\begin{equation*}
    \forall z,w\in\Omega,\quad \mathbb{E}[X(z)X(w)]\eqqcolon K_{\text{N}}(z,w)=-\ln|z-w||z-\overline{w}|+g(z,w).
\end{equation*}
The exact correction term $g(z,w)$ can be calculated with the Green function associated to the Neumann boundary condition, but we don't need its exact expression here. The hyperbolic Gaussian multiplicative chaos measure is then defined via Equation~\eqref{eq:DefinitionBoundaryGMC}, which we recall here:
\begin{equation*}
    \forall A\subset\Omega,\quad \mu^{\gamma}_{\text{H}}(A)\coloneqq\int_{A}\text{Im}(z)^{-\frac{\gamma^2}{2}}\mu^{\gamma}(dz),
\end{equation*}
with $\mu^{\gamma}$ the classical Gaussian multiplicative chaos measure associated to the log-correlated Gaussian field $X$, whose definition is summerized in the previous section.

\begin{rema}[On the extended Seiberg bounds for the boundary Liouville conformal field theory]\label{rema:ShortPhysics}
We give a brief overview of origin of the extended Seiberg bounds in boundary Liouville conformal field theory and its connection to the moments of the hyperbolic Gaussian multiplicative chaos measure above. This remark can be safely skipped, as Theorem~\ref{th:Main1} can be understood as a purely probabilistic result.

The law of the boundary Liouville conformal field theory, in the unit disk (or equivalently on the half-plane via a standard conformal transformation), can be described in terms of the joint law of two random measures $(Z_0,Z_0^{\partial})$ where $Z_0$ is the ``bulk'' measure and the $Z_0^{\partial}$ the ``boundary'' measure. Especially, the bulk measure $Z_0$, up to some transformations and correction terms, is written as
\begin{equation*}
    dZ_0(d^2z)=e^{\gamma G(z)}d\mu^{\gamma}_{\text{H}}(d^2z)
\end{equation*}
with a singularity kernel $G(z)=\sum_{i}\alpha_i K(z_i,z)+\sum_{j}\beta_j K(s_j,z)$, induced by marked points $\{z_i;\alpha_i\}_{i}$ inside the disk and $\{s_j;\beta_j\}_{j}$ on the boundary. To be precise, one should do everything with the geometry of the unit disk to set the constants correct, but for simplicity we refer to~\cite[Section~3.6]{huang2018liouville} for details on the unit disk setting. The only relevant fact for us is that the correct bulk measure $\mu^{\gamma}_{\text{H}}$ is given by Equation~\eqref{eq:DefinitionBoundaryGMC}, reweighted possibly by some point-like singularities.

A notable and simple example is that the Gaussian multiplicative chaos expression for the boundary Liouville correlation functions in the unit volume case~\cite[Corollary~3.8]{huang2018liouville} has the following renormalization constant:
\begin{equation*}
    \mathcal{Z}\coloneqq\Gamma(s)\mathbb{E}[Z_0(\mathbb{D})^{-s}],
\end{equation*}
where $s=\frac{\sum_i\alpha_i+\frac{1}{2}\sum_j\beta_j-Q}{\gamma}$, where we set the bulk cosmological constant equal to $1$ (and the boundary cosmological constant to $0$) for simplicity.

We observe that, up to singularities given by special functions (here the Gamma function), the probabilistic obstruction for the existence of boundary Liouville correlation functions is a moment of the Gaussian multiplicative chaos measure, possibly with extra marked points. The extended Seiberg bound, in the probabilistic language of Gaussian multiplicative chaos, is thus equivalent to a combination of Theorem~\ref{th:Main1} and~\cite[Lemma~3.10]{david2016liouville}. More precisely, Theorem~\ref{th:Main1} of this current paper handles the case with the presence of the $1$-dimensional boundary singularity but no insertion points, while the argument~\cite[Lemma~3.10]{david2016liouville} can then be invoked for adding discrete $0$-dimensional singularities, i.e. the singularity kernel term $G$ above. For more on the physics background and different formulations of the theorem, we refer to \cite[Section~3.6]{huang2018liouville}. Especially, Theorem~\ref{th:Main1} implies that the conditions given in~\cite[Corollary~3.10]{huang2018liouville}, which are called the unit volume Seiberg bounds, are optimal under the Gaussian multiplicative chaos construction.
\end{rema}

\subsection{Exact boundary scaling kernel}
It is a well-known observation since Kahane~\cite{kahane1985chaos} that the existence of moments for the Gaussian multiplicative chaos measure does not depend on the correction term $g(z,w)$. This is quantified via Kahane's convexity inequality, which we recall in Lemma~\ref{lem:KahaneConvexity}. Consequently, for the purpose of this paper, it is enough (and more convenient) to work with the following so-called exact boundary scaling log-correlated kernel
\begin{equation}\label{eq:ExactScalingBoundaryKernel}
    \mathbb{E}[X(z)X(w)]=K_{\mathbb{H}}(z,w)=-\ln|z-w||z-\overline{w}|,
\end{equation}
where $\overline{w}$ is the complex conjugate of $w$ in $\mathbb{C}$. Kahane's convexity inequality yields that, for any correction term $g(z,w)$ bounded in absolute value, such that for any Carleson cube $Q\in\mathcal{Q}^{\ast}_{\mathbb{R}}$ and any $p\in\mathbb{R}$, the $p$-th moments
\begin{equation*}
    \mathbb{E}[\mu^{\gamma}_{\text{H}}(Q)^{p}]
\end{equation*}
associated with the two log-correlated Gaussian fields (one with correction term $g(z,w)$ and the other without) are comparable, up to some multiplicative constant $C$ depending only on $p,\gamma$ and $||g(z,w)||_{\infty}$, whenever any of these moments exists. See Lemma~\ref{lem:KahaneConvexity} for Kahane's convexity inequality and Corollary~\ref{coro:KahaneApplication} for a quick proof of this application.

Since we focus mainly on the existence of moments for Gaussian multiplicative chaos measures, in the rest of the paper, we work with the exact boundary scaling kernel of~\eqref{eq:ExactScalingBoundaryKernel} unless mentioned otherwise.

\subsection{Boundary scaling relation}
We now review some basic multifractal properties of the hyperbolic Gaussian multiplicative chaos measures with the exact boundary scaling kernel~\eqref{eq:ExactScalingBoundaryKernel}. In the sequel, we use
\begin{equation}\label{eq:ScalingZetaFactor}
    \overline{\zeta}(p)\coloneqq \left(2+\frac{\gamma^2}{2}\right)p-\gamma^2p^2
\end{equation}
to denote the boundary scaling exponent in the following lemma and afterwards.
\begin{lemm}[Hyperbolic exact scaling relation]\label{lem:HyperbolicScaling}
    Consider a domain $A\subset\Omega$ where $\Omega\subset\overline{\mathbb{H}}$ is a region containing the origin such that the covariance kernel~\eqref{eq:ExactScalingBoundaryKernel} is well-defined. For $0<r<1$, denote by $rA$ the domain of $\overline{\mathbb{H}}$ that is obtained by shrinking $A$ with respect to the origin,
    \begin{equation*}
        rA\coloneqq\{w\in\overline{\mathbb{H}}~;~\frac{1}{r}w\in A\}.
    \end{equation*}
    Then for any $p\in\mathbb{R}$ such that the moment $\mathbb{E}[\mu^{\gamma}_{\text{H}}(A)^{p}]$ exists, we have
    \begin{equation*}
        \mathbb{E}[\mu^{\gamma}_{\text{H}}(rA)^{p}]=r^{\overline{\zeta}(p)}\mathbb{E}[\mu^{\gamma}_{\text{H}}(A)^{p}],
    \end{equation*}
    where $\overline{\zeta}(p)$ is the scaling factor defined in~\eqref{eq:ScalingZetaFactor}.
\end{lemm}

Before entering the proof, let us introduce properly the hyperbolic circle-average that is somewhat more adapted to the boundary Gaussian multiplicative chaos. For $z\in\mathbb{H}$, let $B(z,\epsilon)$ be the hyperbolic circle of radius $\epsilon$, and let $X_{\epsilon}(z)$ be the average of the log-correlated field $X$ over the circle $B(z,\epsilon)$ weighted by the hyperbolic metric. This regularization has the advantage that $X_{\epsilon}(z)$ is defined simultaneously for all $z\in\mathbb{H}$, while with the usual Euclidean circle-average, we have to restrict ourselves to the set that is $\epsilon$-away from the boundary $\mathbb{R}=\partial\mathbb{H}$. In practice this makes no difference, but we feel healthy to fix this convention for the rest of the article to avoid possible confusions.

\begin{proof}
    For $z,w\in A$, we have that $rz,rw\in rA$ and
    \begin{equation*}
        K(rz,rw)=-2\ln r+K(z,w).
    \end{equation*}
    One verifies, by standard calculation on hyperbolic distances, the following equality in law for the regularized Gaussian fields:
    \begin{equation*}
        \{X_{\epsilon}(z)+\sqrt{-2\ln r}N\}_{z\in A}\stackrel{\text{(law)}}{=}\{X_{r\epsilon}(rz)\}_{z\in A},
    \end{equation*}
    where $N$ is an independent standard Gaussian variable. It follows that
    \begin{equation*}
    \begin{split}
        \mathbb{E}[\mu^{\gamma}_{\text{H},r\epsilon}(rA)^{p}]&=\mathbb{E}\left[\left(\int_{rA}e^{\gamma X_{r\epsilon}(z)-\frac{\gamma^2}{2}\mathbb{E}[X_{r\epsilon}(z)^2]}\text{Im}(z)^{-\frac{\gamma^2}{2}}d^2z\right)^{p}\right]\\
        &=\mathbb{E}\left[\left(\int_{A}e^{\gamma X_{\epsilon}(z)-\frac{\gamma^2}{2}\mathbb{E}[X_{\epsilon}(z)^2]}e^{\gamma \sqrt{-2\ln r}N-\frac{\gamma^2}{2}\mathbb{E}[(\sqrt{-2\ln r}N)^2]}\text{Im}(rz)^{-\frac{\gamma^2}{2}}d^2(rz)\right)^{p}\right]\\
        &=r^{-\gamma^2p^2+\gamma^2 p}\cdot r^{-\frac{\gamma^2}{2}p+2p}\cdot\mathbb{E}\left[\left(\int_{A}e^{\gamma X_{\epsilon}(z)-\frac{\gamma^2}{2}\mathbb{E}[X_{\epsilon}(z)^2]}\text{Im}(z)^{-\frac{\gamma^2}{2}}d^2z\right)^{p}\right]\\
        &=r^{\overline{\zeta}(p)}\mathbb{E}[\mu^{\gamma}_{\text{H},\epsilon}(A)^{p}],
    \end{split}
    \end{equation*}
    with $\overline{\zeta}(p)$ as in~\eqref{eq:ScalingZetaFactor}. The claim follows then from approximation and passing to the $\epsilon\to0$ limit.
\end{proof}

The following check is elementary:
\begin{prop}[Sign-change of the hyperbolic scaling exponent]\label{prop:TrivialProposition}
    Let $\gamma\in(0,2)$. The function $p\mapsto 1-\overline{\zeta}(p)$ is strictly negative if and only if $\frac{1}{2}<p<\frac{2}{\gamma^2}$.
\end{prop}

\section{Setup and strategy}\label{sec:Strategy}

\subsection{Setup and heuristics for the threshold}
Before presenting the details of the proof, we give an overview of the strategy. The key idea is to divide $Q_r=Q_{[-r,r]}$ into three parts \`{a} la Whitney, in view of applying the boundary scaling relation of Lemma~\ref{lem:HyperbolicScaling}. The decomposition is illustrated in Figure~\ref{fig:LowerParts}.

\begin{figure}[h]
\centering
\includegraphics[height=15em]{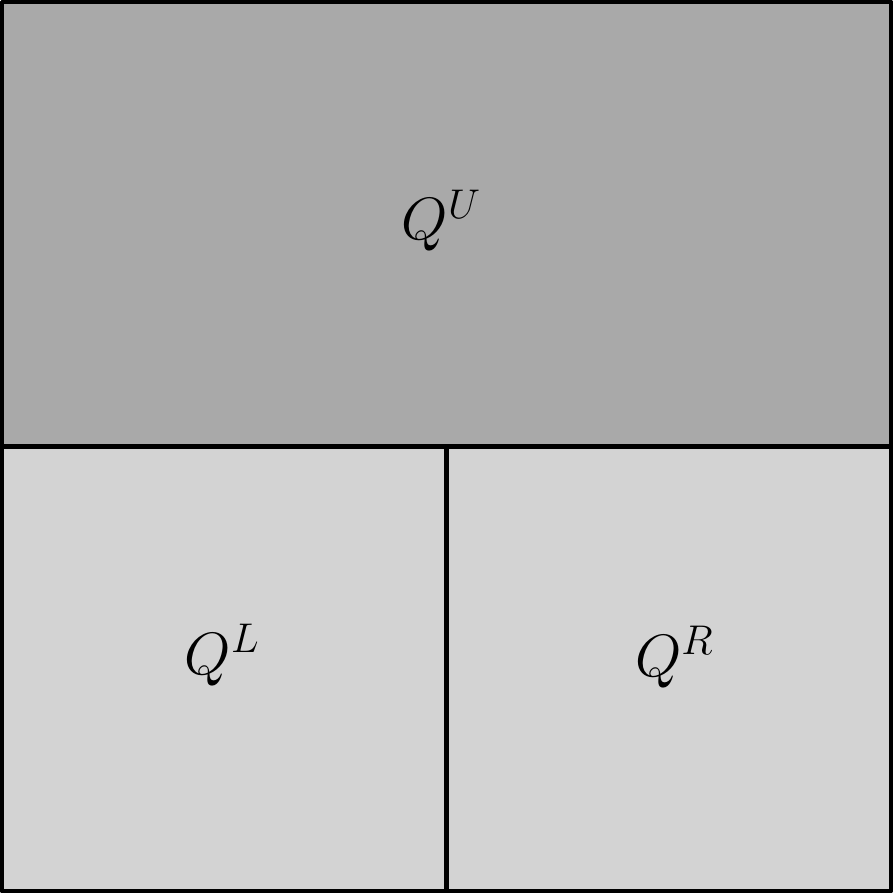}
\caption{Dividing a Carleson cube into three parts \`{a} la Whitney.}
\label{fig:LowerParts}
\end{figure}

More precisely, with ``$\text{L}, \text{R}, \text{U}$'' respectively stands for ``left, right, upper'', write $Q^{\text{L}}_r=Q_{[-r,0]}$, $Q^{\text{R}}_r=Q_{[0,r]}$ and $Q^{\text{U}}=Q_{[-r,r]}\setminus(Q^{\text{L}}_r\cup Q^{\text{R}}_{r})$, so that
\begin{equation}
    Q_r=Q^{\text{L}}_r\cup Q^{\text{R}}_r\cup Q^{\text{U}}_r
\end{equation}
is a partition of $Q_r$ into smaller rectangles (to be precise, we ignore intersections of these rectangles as the Gaussian multiplicative chaos measure is almost surely null there).

\medskip

Several preliminary observations are in order:
\begin{enumerate}
    \item The moments of any of the lower parts (e.g. of $Q^{\text{L}}_r$) are related to that of $Q_r$, by means of the boundary scaling relation of Lemma~\ref{lem:HyperbolicScaling}.
    \item The Gaussian multiplicative chaos measures of the lower parts behave almost independently (if we treat the region where they meet appropriately), since the underlying Gaussian fields are log-correlated and does not exhibit long-range correlations;
    \item The Gaussian multiplicative chaos measure of the upper part $Q^{\text{U}}_r$ is not responsible for the explosion of moment: indeed, it behaves as a classical Gaussian multiplicative chaos without singularity and should have higher moment bound.
\end{enumerate}
These observations serve as the basis for the proof of the extended Seiberg bound in the case $\gamma\in(0,\sqrt{2})$. We are lead to the following heuristic calculation when $p$ is close to the moment bound threshold:
\begin{equation*}
    \mathbb{E}[\mu^{\gamma}_{\text{H}}(Q_r)^{p}]\approx\mathbb{E}[\mu^{\gamma}_{\text{H}}(Q^{\text{L}}_r+Q^{\text{R}}_r)^{p}]\stackrel{(\star)}{\approx}\mathbb{E}[\mu^{\gamma}_{\text{H}}(Q^{\text{L}}_r)^{p}]+\mathbb{E}[\mu^{\gamma}_{\text{H}}(Q^{\text{R}}_r)^{p}]\approx 2^{1-\overline{\zeta}(p)}\mathbb{E}[\mu^{\gamma}_{\text{H}}(Q^{\text{U}}_r)^{p}].
\end{equation*}
The first approximation results from the last item in the observations above, and the last approximation from the first item. The technical difficulty is to rigorously quantify the approximation $(\star)$ in the middle of the above heuristic calculation. Indeed, if the random variables $Q^{\text{L}}_r$ and $Q^{\text{R}}_r$ were highly correlated -- think of the extreme case where they are identical -- then the approximation $(\star)$ would come with an extra multiplicative factor $2^{p}$. That is, we should obtain a Jensen-type estimate in the highly correlated case. Based on the middle item of the observations prior to $(\star)$, we should expect the emergence of a totally opposite behavior near the critical moment bound, resulting in the additive-type estimate in the direction of $(\star)$. We will use a combination of Gaussian decorrelation inequalities, elementary inequalities, some tail estimates together with a tailored combinatorial Sokoban lemma in order to show that, as $p$ is close to the critical moment bound $p_c$, the approximation $(\star)$ is almost satisfied up to small corrections.

In the case with large coupling constant $\gamma\in[\sqrt{2},2)$, the above heuristic is still our guiding philosophy, although the technical details are different. This is in part because of the following reasons, all of which result from the peculiar property of the boundary Gaussian multiplicative chaos, which now fails to possess a finite first moment. First, we lose the Banach space structure of $L^{p}$ spaces for $p>1$. Second, it turns out that, due to the duality between the subadditive/superadditive-type estimates, the easy part in the proof for small coupling constant $\gamma\in(0,\sqrt{2})$ becomes the difficult part in the proof for large coupling constant $\gamma\in[\sqrt{2},2)$, and vice versa. Lastly, the meticulous manipulations of~\cite{kahane1976certaines}, based on the recurrence scheme over $p\in]k,k+1]$ for integer $k\geq 1$ cannot be applied, now that we don't have the $p=1$ case to initiate our induction.

\medskip

In spite of these technical difficulties to justify our scheme, if we admit these approximations for now, we observe that the system goes through phase transition at thresholds $p$ satisfying the equation given by Proposition~\ref{prop:TrivialProposition},
\begin{equation*}
    0=1-\overline{\zeta}(p)=1-(2+\frac{\gamma^2}{2})p+\gamma^2 p^2.
\end{equation*}
Solving this yields two solutions: $p=\frac{1}{2}$ or $p=\frac{2}{\gamma^2}$, the latter being the predicted extended Seiberg bound for the boundary Liouville conformal field theory in~\cite{huang2018liouville}, which is restated in the form of Theorem~\ref{th:Main1}.

\begin{rema}
Naturally, the techniques in the paper provide an alternative and self-contained proof to Kahane's positive moment bound on the classical Gaussian multiplicative chaos. The main input of this paper is to provide a framework to treat moment bound problems for Gaussian multiplicative chaos with arbitrary singularities of higher dimensions and with more general symmetries, which is not investigated in the classical theory. Our method is general enough to treat singularities of any higher (co)-dimension, but for simplicity of the presentation we mainly focus on the disk model now, i.e. $1$-dimensional singularity embedded in a $2$-dimensional manifold, and give details to the general case in a future work.
\end{rema}

\subsection{Structure of the proof}
We briefly describe the structure of the rest of this article.

We start by separating two regimes: the regime for small coupling constant $\gamma\in(0,\sqrt{2})$, where the expectation of $\mu^{\gamma}_{\text{H}}(Q)$ is finite for any Carleson cube $Q$; and the regime for large coupling constant $\gamma\in[\sqrt{2},2)$, where the expectation of $\mu^{\gamma}_{\text{H}}(Q)$ explodes for any Carleson cube $Q$.

In both of these regimes, in order to show that the bound in Theorem~\ref{th:Main1} is optimal, we have to prove two directions. The subadditive inequality exploited in Section~\ref{sec:Subadditivity} deals with one direction in the regime for large $\gamma\in[\sqrt{2},2)$, while its dual counterpart, the superadditive inequality exploited in Section~\ref{sec:Superadditivity}, deals with one direction in the regime for small $\gamma\in(0,\sqrt{2})$.

The main technical difficulties of the proofs are on proving the optimality of the bounds above, i.e. proving the converses of the above directions in respective regimes. The sufficient condition for the existence of boundary Liouville correlation functions in the regime for small $\gamma\in(0,\sqrt{2})$ is the most physically relevant one. This part is previously treated in an unpublished appendix of~\cite{huang2018liouville}, by using a combination of Gaussian decorrelation inequalities and combinatorial observations. We present a polished, pedagogical version of this proof in Section~\ref{sec:DecorrelationPlusSokoban}.

The last piece of the puzzle, in order to establishing the full optimality of the extended Seiberg bound for boundary Liouville conformal field theory, is the necessary condition for the existence of boundary Liouville correlation functions in the regime for large coupling constant $\gamma\in[\sqrt{2},2)$. This is the main contribution of this paper, with several extra difficulties coming from the lack of a finite first moment, since the expectation of $\mu^{\gamma}_{\text{H}}(Q)$ explodes for any Carleson cube $Q$. One important ingredient is to show that the moment of $\mu^{\gamma}_{\text{H}}(Q)$ near the critical moment bound cannot be too small, by means of a no exponential decay lemma at the critical threshold. We present the relevant proof in Section~\ref{sec:HardPart}.

Finally, for better readability, some technical estimates related to Gaussian decorrelation techniques are gathered and detailed in Section~\ref{sec:TechnicalEstimates}.

\section{Proof of the main theorem for small moments}
We start with the unconventional case $\gamma\in[\sqrt{2},2)$, where the first moment of the hyperbolic Gaussian multiplicative chaos measure is infinite:
\begin{equation}\label{eq:ExplosionFirstMoment}
    \forall Q\in\mathcal{Q}^{*}_{\mathbb{R}}, \quad \mathbb{E}\left[\mu_{\text{H}}(Q)\right]=\mathbb{E}\left[\int_{Q}\text{Im}(z)^{-\frac{\gamma^2}{2}}d^2z\right]=\infty.
\end{equation}
We will show that, for any $Q\in\mathcal{Q}^{*}_{\mathbb{R}}$,
\begin{enumerate}
    \item The random measure $\mu_{\text{H}}(Q)$ is almost surely non-trivial;
    \item The random measure $\mu_{\text{H}}(Q)$ has finite $\frac{1}{2}$-moment for all $\gamma\in[\sqrt{2},2)$;
    \item The random measure $\mu_{\text{H}}(Q)$ has finite $p$-th moment if and only if $p<\frac{2}{\gamma^2}$.
\end{enumerate}
The statement in the last item is the extended Seiberg bound for the boundary Liouville conformal field theory in the unit disk, stated in the form of Theorem~\ref{th:Main1} (here with $\gamma\in[\sqrt{2},2)$). Previously in an unpublished version of~\cite{huang2018liouville}, only the sufficient condition (i.e. the ``if'' part) of the last item was proven. As mentioned in the introduction, this sufficient condition was proven to be a useful ingredient in many applications in boundary Liouville conformal field theory, but the optimality of this bound was left unsettled. In this section, we establish the converse statement (i.e. the ``only if'' part of the last item above).

\subsection{Existence of small moments}\label{sec:Subadditivity}
We first establish the existence of small moments, which also guarantees the non-trivialness of the measure near the boundary.
\begin{lemm}[Finiteness of small positive moments]\label{lem:ExistenceHalfMoment}
    Let $\gamma\in[\sqrt{2},2)$ and consider $\mu_{\text{H}}(Q)$ with $Q\in\mathcal{Q}^{*}_{\mathbb{R}}$ as above. Then for any $p<\frac{2}{\gamma^2}$, the moment
    \begin{equation*}
        \mathbb{E}[\mu_{\text{H}}(Q)^{p}]
    \end{equation*}
    exists and is finite.
\end{lemm}
\begin{proof}
    Without loss of generality, take $Q=Q_{r}=Q_{[-r,r]}$ and consider the Whitney-type decomposition explained in Section~\ref{sec:Strategy}, with $Q^{\text{L}}_r, Q^{\text{R}}_r$ the smaller boundary Carleson cubes and $Q^{\text{U}}_r$ their complement in the upper half of $Q_{r}$.

    The statement for the negative moments $p<0$ is classical: it suffices to notice that, since the measure $\mu_{\text{H}}(Q^{\text{U}}_r)$ is a classical Gaussian multiplicative chaos and $\gamma<2$, it has negative moments of all order~\cite[Theorem~2.12]{Rhodes_2014}. It follows then $\mu_{\text{H}}(Q_r)$ has negative moments of all order.

    We next consider positive moments of $\mu_{\text{H}}(Q_r)$ for $p$ close to $\frac{2}{\gamma^2}$. So let $\frac{1}{2}<p<\frac{2}{\gamma^2}$. We can keep on iterating the Whitney decomposition for the smaller boundary Carleson cubes and their children: we obtain a partition of $Q_r$ by a family of rectangles of the same shape as $Q^{\text{U}}_r$, but with different sizes. More precisely, there will be $2^{n}$ such rectangles of size $2^{-n}Q^{\text{U}}_r$ in the above partition. See Figure~\ref{fig:InfiniteWhitney} for a graphical representation.

\begin{figure}[h]
\centering
\includegraphics[height=15em]{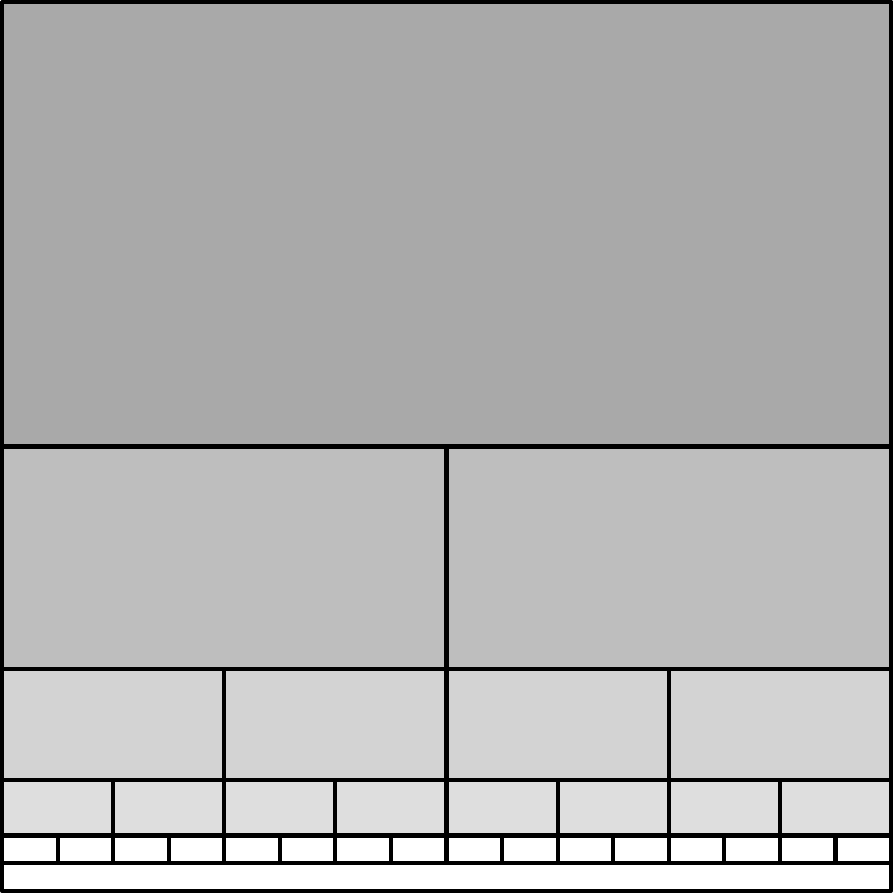}
\caption{Dividing a Carleson cube into infinitely many smaller rectangles.}
\label{fig:InfiniteWhitney}
\end{figure}

    By the boundary scaling relation of Lemma~\ref{lem:HyperbolicScaling}, such a rectangle of size $2^{-n}Q^{\text{U}}_r$ has $p$-th moment equal to $2^{-n\overline{\zeta}(p)}\mathbb{E}[\mu_{\text{H}}(Q^{\text{U}}_r)^{p}]$. Therefore, by subadditivity, we can control the $p$-th moment of $\mu_{\text{H}}(Q_r)$ by the sum of those moments over the rectangles in the above partition:
    \begin{equation*}
        \mathbb{E}[\mu_{\text{H}}(Q_r)^{p}]\leq\sum\limits_{n=0}^{\infty}2^{n}\cdot2^{-n\overline{\zeta}(p)}\mathbb{E}[\mu_{\text{H}}(Q^{\text{U}}_r)^{p}]=\frac{1}{1-2^{1-\overline{\zeta}(p)}}\mathbb{E}[\mu_{\text{H}}(Q^{\text{U}}_r)^{p}].
    \end{equation*}
    The last term is finite when $\frac{1}{2}<p<\frac{2}{\gamma^2}$ by Proposition~\ref{prop:TrivialProposition}. We also used the fact that the random variable $\mu_{\text{H}}(Q^{\text{U}}_r)$ has positive moments up to order just below $\frac{4}{\gamma^2}$, as a classical Gaussian multiplicative chaos measure~\cite[Theorem~2.11]{Rhodes_2014}. This finishes the proof.
\end{proof}

\begin{coro}[Finiteness of half moment]\label{coro:SmallerThanHalfMoments}
    For any Carleson cube $Q\subset\mathbb{H}$ and any $\gamma\in(0,2)$, the half moment $\mathbb{E}[\mu_{\text{H}}(Q)^{1/2}]$ is always finite.
\end{coro}
\begin{coro}[Non-trivialness of the boundary measure]\label{coro:NonDegenerateMeasure}
    For any Carleson cube $Q\subset\mathbb{H}$ and any $\gamma\in(0,2)$, the random measure $\mu_{\text{H}}(Q)$ is almost surely well-defined and non-trivial, i.e. $0<\mu_{\text{H}}(Q)<\infty$ almost surely.
\end{coro}
By Remark~\ref{rem:CoveringDisk}, the above is true if we replace the Carleson cube $Q$ by any non-empty measurable set of $\overline{\mathbb{H}}$ (or in the disk model, of $\overline{\mathbb{D}}$).

\subsection{Determining the exact moment bound}\label{sec:HardPart}
Without loss of generality, fix $r>0$ and let $Q_r=Q_{[-r,r]}$. We know from \eqref{eq:ExplosionFirstMoment} and Lemma~\ref{lem:ExistenceHalfMoment} that in the $\gamma\in[\sqrt{2},2)$ regime, the following threshold exists:
\begin{equation}
    p_c(\gamma)\coloneqq\sup\left\{p\geq\frac{1}{2}~;~\mathbb{E}[\mu_{\text{H}}(Q_r)^{p}]<\infty\right\}=\inf\{p\leq 1~;~\mathbb{E}[\mu_{\text{H}}(Q_r)^{p}]=\infty\}\in\left[\frac{2}{\gamma^2},1\right].
\end{equation}

We first give a crude estimate to improve the upper bound on $p_c(\gamma)$.
\begin{lemm}\label{lem:BadEstimate}
The threshold $p_c(\gamma)$ is bounded away from $1$. More precisely,
\begin{equation*}
    \forall \gamma\in[\sqrt{2},2),\quad p_c(\gamma)\leq\frac{1}{2}+\frac{1}{\gamma^2}.
\end{equation*}
In particular, for $\gamma\in(\sqrt{2},2)$, $p_c(\gamma)<1$.
\end{lemm}
\begin{proof}
The case for $\gamma=\sqrt{2}$ is clear: by~\eqref{eq:ExplosionFirstMoment}, $p_c(\sqrt{2})=1$ and the moment at $p_c(\sqrt{2})$ explodes.

Therefore, take $\gamma\in(\sqrt{2},2)$ and suppose that $\mathbb{E}[\mu_{\text{H}}(Q_r)^{p}]<\infty$ for some $\frac{1}{2}<p<1$. Consider the Whitney decomposition as in Section~\ref{sec:Strategy} (also see Figure~\ref{fig:LowerParts}). By the boundary scaling relation of Lemma~\ref{lem:HyperbolicScaling}, we have
\begin{equation*}
    \mathbb{E}[\mu_{\text{H}}(Q^{\text{L}}_r)^{p}]=\mathbb{E}[\mu_{\text{H}}(Q^{\text{R}}_r)^{p}]=2^{-\overline{\zeta}(p)}\mathbb{E}[\mu_{\text{H}}(Q_r)^{p}]<\infty.
\end{equation*}
Therefore, by Jensen's inequality, we have
\begin{equation*}
    \mathbb{E}[\mu_{\text{H}}(Q_r)^{p}]\geq \mathbb{E}[(\mu_{\text{H}}(Q^{\text{L}}_r)+\mu_{\text{H}}(Q^{\text{R}}_r))^{p}]\geq 2^{p-1}(\mathbb{E}[\mu_{\text{H}}(Q^{\text{L}}_r)^{p}]+\mathbb{E}[\mu_{\text{H}}(Q^{\text{R}}_r)^{p}]),
\end{equation*}
and combining these observations we have
\begin{equation*}
    \mathbb{E}[\mu_{\text{H}}(Q_r)^{p}]\geq 2^{p-\overline{\zeta}(p)}\mathbb{E}[\mu_{\text{H}}(Q_r)^{p}].
\end{equation*}
This implies $p-\overline{\zeta}(p)<0$. A quick calculation with~\eqref{eq:ScalingZetaFactor} shows that $p\leq\frac{1}{2}+\frac{1}{\gamma^2}$.

Therefore, $\mathbb{E}[\mu_{\text{H}}(Q_r)^{p}]<\infty$ implies $p\leq\frac{1}{2}+\frac{1}{\gamma^2}$, thus $p_c(\gamma)\leq\frac{1}{2}+\frac{1}{\gamma^2}$.
\end{proof}

The extended Seiberg bound for the boundary Liouville conformal field theory (with $\gamma\in[\sqrt{2},2)$) is equivalent to the following statement:
\begin{theo}[The extended Seiberg bound for large coupling constants]\label{th:SeibergSmallNecessary}
\begin{equation}
    \forall \gamma\in[\sqrt{2},2),\quad p_c(\gamma)=\frac{2}{\gamma^2}.
\end{equation}
Furthermore, the $p_c(\gamma)$-moment explodes, i.e. $\mathbb{E}[\mu_{\text{H}}(Q_r)^{\frac{2}{\gamma^2}}]=\infty$.
\end{theo}
\begin{rema}\label{rem:NotTheTrivialOne}
The case $\gamma=\sqrt{2}$ is already contained in the observation~\eqref{eq:ExplosionFirstMoment}. We will suppose that $\gamma\in(\sqrt{2},2)$ in the rest of this section. Notice that in this case, Lemma~\ref{lem:BadEstimate} already yields $p_c(\gamma)<1$.
\end{rema}

Before going into the proof, we introduce some relevant quantities that we shall examine. Without loss of generality, consider $Q_r=Q_{[-r,r]}$ and define
\begin{equation*}
    L_{r,n}\coloneqq[-r,r]\times[2^{-n-1}r, 2^{-n}r]\subset Q_r
\end{equation*}
so that $\{L_{r,n}\}_{n\geq 0}$ is a partition of $Q_r$ into horizontal slices, see Figure~\ref{fig:HorizontalRectangles} (again, we ignore boundary where slices intersect). By dominated convergence, to show $\mathbb{E}[\mu_{\text{H}}(Q_r)^{p}]=\infty$, it is sufficient to show that
\begin{equation*}
    \limsup_{n\to\infty}\mathbb{E}[\mu_{\text{H}}(L_{r,n})^{p}]>0.
\end{equation*}

\begin{figure}[h]
\centering
\includegraphics[height=15em]{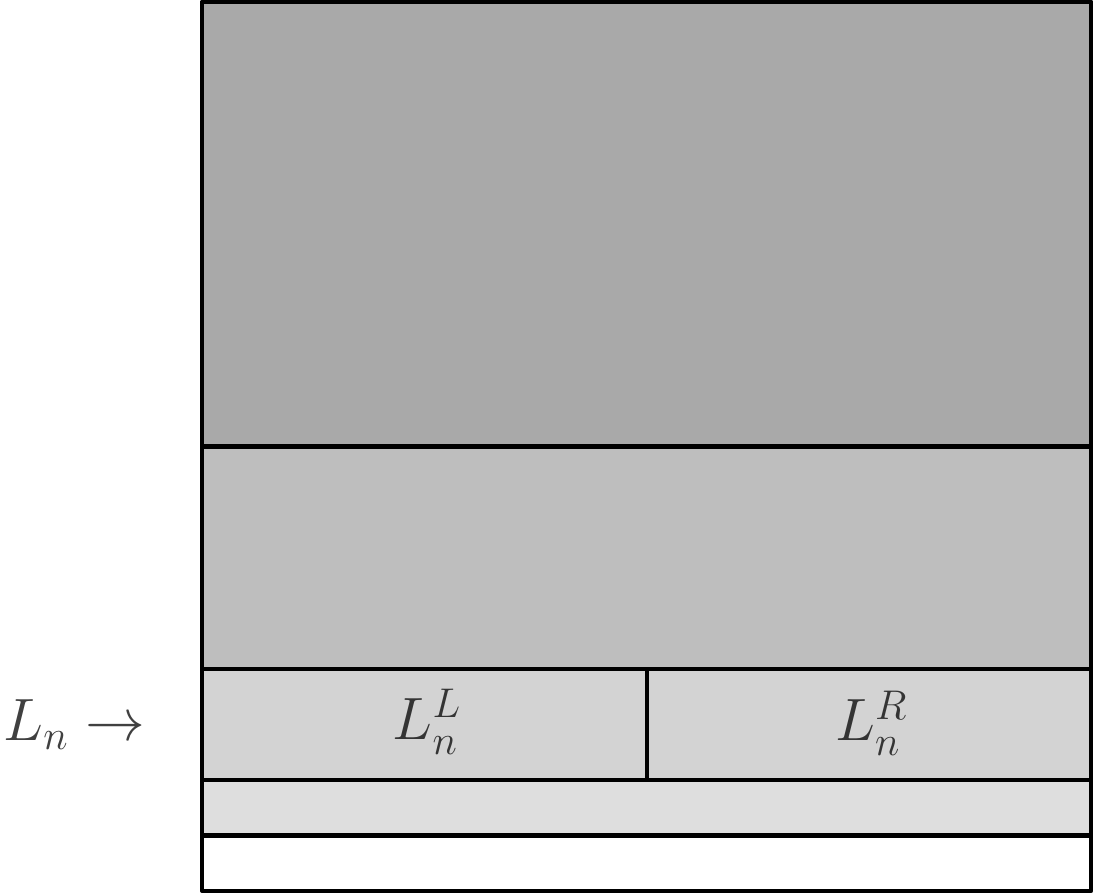}
\caption{Dividing a Carleson cube into horizontal slices.}
\label{fig:HorizontalRectangles}
\end{figure}

We first show the following key lemma, that if $\mathbb{E}[\mu_{\text{H}}(L_{r,n})^{p}]$ exhibits exponential decay in $n$ with strictly positive parameter, then we are away from $p_c(\gamma)$ in the sense that $p_c(\gamma)-p$ is bounded from below by a positive value depending only on this parameter and $\gamma$.
\begin{lemm}[No exponential decay property at $p_c(\gamma)$]\label{lem:NoExponentailDecay}
For any $\kappa>0$, there exists some $p\in(\frac{1}{2},p_c(\gamma))$ close enough to $p_c(\gamma)$ such that
\begin{equation}\label{eq:NoExponentialDecay}
    \limsup_{n\to\infty}2^{\frac{\kappa}{2}n}\mathbb{E}[\mu_{\text{H}}(L_{r,n})^{p}]\geq 1.
\end{equation}
In other words, $\mathbb{E}[\mu_{\text{H}}(L_{r,n})^{p}]$ cannot decay exponentially fast in $n$ with uniformly positive parameter in $p$, when $p$ is approaching the critical parameter $p_c(\gamma)$ from below. 
\end{lemm}
\begin{proof}
We proceed by contradiction, that if~\eqref{eq:NoExponentialDecay} does not hold and $\mathbb{E}[\mu_{\text{H}}(Q_r)^{p}]<\infty$, then for some $\epsilon=\epsilon(\kappa,\gamma)$ independent of $p$, we have $\mathbb{E}[\mu_{\text{H}}(Q_r)^{p+\epsilon}]<\infty$.

For fixed $n$, notice that the first moment of $\mu_{\text{H}}(L_{r,n})$ exists and
\begin{equation}\label{eq:ExponentialFirstMoment}
    \mathbb{E}[\mu_{\text{H}}(L_{r,n})]=\int_{[-r,r]\times[2^{-n-1}r,2^{-n}r]}y^{-\frac{\gamma^2}{2}}dxdy=C(\gamma,r)2^{n(\frac{\gamma^2}{2}-1)},
\end{equation}
where by assumption we are in the $\gamma\in[\sqrt{2},2)$ regime. If~\eqref{eq:NoExponentialDecay} does not hold, then
\begin{equation}\label{eq:DoesNotHold}
     \limsup_{n\to\infty}2^{\frac{\kappa}{4}n}\mathbb{E}[\mu_{\text{H}}(L_{r,n})^{p}]=0.
\end{equation}
By convexity of the function $p\mapsto\ln\mathbb{E}[\mu_{\text{H}}(L_{r,n})^{p}]$, for any $0<\epsilon<1-p$,
\begin{equation*}
    \ln\mathbb{E}[(\mu_{\text{H}}(L_{r,n}))^{p+\epsilon}]\leq\frac{\epsilon}{1-p}\ln\mathbb{E}[\mu_{\text{H}}(L_{r,n})]+\frac{1-p-\epsilon}{1-p}\ln\mathbb{E}[(\mu_{\text{H}}(L_{r,n}))^{p}].
\end{equation*}
Choosing $\epsilon$ small enough such that
\begin{equation}\label{eq:ChoiceEpsilon}
    \eta\coloneqq\frac{\kappa}{4}\left(\frac{1-p-\epsilon}{1-p}\right)-\left(\frac{\gamma^2}{2}-1\right)\frac{\epsilon}{1-p}>0,
\end{equation}
we get from~\eqref{eq:ExponentialFirstMoment} and~\eqref{eq:DoesNotHold} that
\begin{equation*}
    \limsup_{n\to\infty}2^{\eta n}\mathbb{E}[(\mu_{\text{H}}(L_{r,n}))^{p+\epsilon}]=0.
\end{equation*}
Together with subadditivity (recall that $\{L_{r,n}\}_{n}$ forms a partition of $Q_r$), this implies that
\begin{equation*}
    \mathbb{E}[\mu_{\text{H}}(Q_r)^{p+\epsilon}]\leq\sum_{n}\mathbb{E}[(\mu_{\text{H}}(L_{r,n}))^{p+\epsilon}]<\infty.
\end{equation*}

Finally, notice that the choice of $\epsilon$ in~\eqref{eq:ChoiceEpsilon} can be made to depend only on $\gamma$ and $\kappa$. Indeed, since we can assume that $p\leq \frac{1}{2}+\frac{1}{\gamma^2}<1$ by Lemma~\ref{lem:BadEstimate}, a possible choice is (also recall that we assume $\gamma\in(\sqrt{2},2)$ by Remark~\ref{rem:NotTheTrivialOne})
\begin{equation*}
    2\epsilon=\frac{\frac{\kappa}{4}\left(\frac{1}{2}-\frac{1}{\gamma^2}\right)}{\frac{\kappa}{4}+\frac{\gamma^2}{2}-1}>0.
\end{equation*}
This means that for $p<p_c(\gamma)$ close enough to $p_c(\gamma)$, namely when $p\in(p_c(\gamma)-\epsilon,p_c(\gamma))$, we must have~\eqref{eq:NoExponentialDecay}. This finishes the proof of the lemma.
\end{proof}

We are now in a position to conclude the proof of the Seiberg bound in the large coupling constant regime $\gamma\in[\sqrt{2},2)$. We will make use of a Gaussian decorrelation inequality recorded in Corollary~\ref{coro:AlmostFactorization}, the proof of which demands some rather lengthy development and will be detailed in Section~\ref{sec:GaussianEstimates}.

\begin{proof}[Proof of Theorem~\ref{th:SeibergSmallNecessary}]
We proceed in two steps.

\medskip

$\bullet$ We first show the weaker statement that $p_c(\gamma)=\frac{2}{\gamma^2}$ by showing that
\begin{equation*}
    \forall p>\frac{2}{\gamma^2},\quad \limsup_{n\to\infty}\mathbb{E}[\mu_{\text{H}}(L_{r,n})^{p}]>0.
\end{equation*}
To this end, we use the boundary scaling relation of Lemma~\ref{lem:HyperbolicScaling} and Gaussian decorrelation inequalities to compare $\mathbb{E}[\mu_{\text{H}}(L_{r,n})^{p}]$ and $\mathbb{E}[\mu_{\text{H}}(L_{r,n+1})^{p}]$. First, we divide $L_{r,n+1}$ into its left and right parts (see Figure~\ref{fig:HorizontalRectangles}),
\begin{equation*}
    L_{r,n+1}=[-r,0]\times[2^{-n-2}r, 2^{-n-1}r]\cup [0,r]\times[2^{-n-2}r, 2^{-n-1}r]\eqqcolon L^{\text{L}}_{r,n+1}\cup L^{\text{R}}_{r,n+1}.
\end{equation*}
The boundary scaling relation of Lemma~\ref{lem:HyperbolicScaling} then implies
\begin{equation*}
    \mathbb{E}[\mu_{\text{H}}(L^{L}_{r,n+1})^{p}]=\mathbb{E}[\mu_{\text{H}}(L^{R}_{r,n+1})^{p}]=2^{-\overline{\zeta}(p)}\mathbb{E}[\mu_{\text{H}}(L_{r,n})^{p}].
\end{equation*}
An elementary equality recalled in Appendix~\ref{sec:Elementary}, namely
\begin{equation*}
    \forall 0<p'<p<1,\quad x^{p}+y^{p}\leq (x+y)^{p}+2(x^{p'}y^{p-p'}+x^{p-p'}y^{p'}),
\end{equation*}
implies that
\begin{equation*}
    \mathbb{E}[\mu_{\text{H}}(L_{r,n+1})^{p}]\geq 2^{1-\overline{\zeta}(p)}\mathbb{E}[\mu_{\text{H}}(L_{r,n})^{p}]-4\mathbb{E}[\mu_{\text{H}}(L^{\text{L}}_{r,n+1})^{p'}\mu_{\text{H}}(L^{\text{R}}_{r,n+1})^{p-p'}],
\end{equation*}
where the factor $4$ in the last term comes from the symmetry between $L^{\text{L}}_{r,n+1}$ and $L^{\text{R}}_{r,n+1}$.

By the Gaussian decorrelation inequality of Corollary~\ref{coro:AlmostFactorization}, for any $\delta>0$,
\begin{equation*}
\begin{split}
    &\mathbb{E}[\mu_{\text{H}}(L^{\text{L}}_{r,n+1})^{p'}\mu_{\text{H}}(L^{\text{R}}_{r,n+1})^{p-p'}]\\
    \leq{}&C_{p}\delta^{p-p'}\mathbb{E}[\mu_{\text{H}}(L^{L}_{r,n+1})^{p}]+\delta^{-2\gamma^2p'(p-p')}\mathbb{E}[\mu_{\text{H}}(L^{\text{L}}_{r,n+1})^{p'}]\mathbb{E}[\mu_{\text{H}}(L^{\text{R}}_{r,n+1})^{p-p'}]\\
    ={}&C_{p}\delta^{p-p'} 2^{-\overline{\zeta}(p)}\mathbb{E}[\mu_{\text{H}}(L_{r,n})^{p}]+\delta^{-2\gamma^2p'(p-p')}\mathbb{E}[\mu_{\text{H}}(L^{\text{L}}_{r,n+1})^{p'}]\mathbb{E}[\mu_{\text{H}}(L^{\text{R}}_{r,n+1})^{p-p'}].
\end{split}
\end{equation*}
By assumption, $p>\frac{2}{\gamma^2}$, so that by Proposition~\ref{prop:TrivialProposition} we can choose some $\delta>0$ such that $u\coloneqq 2^{1-\overline{\zeta}(p)}-4C_p\delta^{p-p'}2^{-\overline{\zeta}(p)}>1$, independent of $n$. By the same argument as in the proof of Lemma~\ref{lem:ExistenceHalfMoment}, we also know that if we choose $p'<p$ such that $\frac{1}{2}<p'<\frac{2}{\gamma^2}$
\begin{equation*}
    v_n\coloneqq\mathbb{E}[\mu_{\text{H}}(L^{\text{L}}_{r,n+1})^{p'}]\leq \frac{1}{2}2^{(n+1)(1-\overline{\zeta}(p'))}\mathbb{E}[\mu_{\text{H}}(Q^{\text{U}}_{r})^{p'}]\leq C2^{-\kappa n},
\end{equation*}
where $\kappa=\kappa(p')=-(1-\overline{\zeta}(p'))>0$ by Proposition~\ref{prop:TrivialProposition}, independent of $n$. Observe that $\kappa$ can be chosen independent of $\frac{2}{\gamma^2}<p$, since it only depends on our choice of $p'$. Notice also that necessarily $0<p-p'<\frac{1}{2}$, so that
\begin{equation*}
    \mathbb{E}[\mu_{\text{H}}(L^{\text{R}}_{r,n+1})^{p-p'}]\leq \mathbb{E}[\mu_{\text{H}}(Q)^{p-p'}]<\infty
\end{equation*}
is bounded uniformly in $n$ by Corollary~\ref{coro:SmallerThanHalfMoments}.

Summerizing, for some fixed $\delta>0$ and some constants $C>0, \kappa>0, u>1$ all independent of $n$, we have
\begin{equation}\label{eq:IterationRelation}
    \mathbb{E}[\mu_{\text{H}}(L_{r,n+1})^{p}]\geq u\mathbb{E}[\mu_{\text{H}}(L_{r,n})^{p}]-C2^{-\kappa n}.
\end{equation}
By Lemma~\ref{lem:NoExponentailDecay} and the observation that $\kappa$ can be chosen independent of $p$, if $p$ is close enough to $p_c(\gamma)$, the factor $Ce^{-\kappa n}$ is negligible in front of $\mathbb{E}[\mu_{\text{H}}(L_{r,n})^{p}]$ for large enough $n$. Since we have chosen $\delta>0$ small enough such that $u>1$, a simple iteration of relation~\eqref{eq:IterationRelation} yields
\begin{equation*}
    \limsup_{n\to\infty}\mathbb{E}[\mu_{\text{H}}(L_{r,n})^{p}]>0,
\end{equation*}
which is the contradiction that we are after. This completes the proof of $p_c(\gamma)=\frac{2}{\gamma^2}$.

\medskip

$\bullet$ We now consider the case of the critical threshold $p=p_c=\frac{2}{\gamma^2}$, where the above proof does not apply directly. Indeed, there is not enough space for the choice of a uniform cutoff $\delta>0$ as $1-\overline{\zeta}(\frac{2}{\gamma^2})=0$ by Proposition~\ref{prop:TrivialProposition}. Suitable changes are to be made and we now explain how.

The fact that $v_n\leq C2^{-\kappa n}$ remains intact, where we stress that $\kappa>0$ can be chosen independently of $p,n$ as long as $\gamma\in(\sqrt{2},2)$. Repeating the argument above, for any sequence $\delta_n>0$, define the associated sequence $u_n\coloneqq 2^{1-\overline{\zeta}(p_c)}-4C_p\delta^{p-p'}2^{-\overline{\zeta}(p)}=1-C_p\delta_n^{p-p'}$ (where we redefined $C_p$ for simplicity) and we have
\begin{equation*}
    \mathbb{E}[\mu_{\text{H}}(L_{r,n+1})^{p}]\geq u_n\mathbb{E}[\mu_{\text{H}}(L_{r,n})^{p}]-C\delta_n^{-2\gamma^2p'(p-p')}2^{-\kappa n}.
\end{equation*}
The argument for the no exponential decay property of~\eqref{eq:NoExponentialDecay} still holds for $p=p_c$, so
\begin{equation*}
    \limsup_{n\to\infty}2^{\frac{\kappa}{2}n}\mathbb{E}[\mu_{\text{H}}(L_{r,n})^{p}]\geq 1.
\end{equation*}
Therefore, we can fix $2p'=\frac{1}{2}+\frac{2}{\gamma^2}$ so that if we choose $\delta_n^{-2\gamma^2p'(p-p')}=O(2^{\frac{\kappa}{4}n})$, for example $\delta_n=2^{-\eta n}$ with $\eta=\frac{\kappa}{8\gamma^2p'(p-p')}>0$, then for large enough $n$, 
\begin{equation}\label{eq:FinalIteration}
    \mathbb{E}[\mu_{\text{H}}(L_{r,n+1})^{p}]\geq u_n\mathbb{E}[\mu_{\text{H}}(L_{r,n})^{p}]-\frac{1}{2}2^{-\frac{2\kappa}{3}n}.
\end{equation}
The factor $\frac{1}{2}$ in the last item above is only added for convenience in the sequel, it can be absorbed in the factor $2^{-\frac{2\kappa}{3}n}$ when $n$ is large enough.

Again by Lemma~\ref{lem:NoExponentailDecay}, for some large $n_0$, $\mathbb{E}[\mu_{\text{H}}(L_{r,n_0})^{p}]\geq 2^{-\frac{\kappa}{2}n_0}$. Notice also that with our choice $2p'=\frac{1}{2}+\frac{2}{\gamma^2}$ and $\gamma\in(\sqrt{2},2)$, $p-p'=\frac{2}{\gamma^2}-p'>0$. Then
\begin{equation*}
    \mathbb{E}[\mu_{\text{H}}(L_{r,n_{0}+1})^{p}]\geq (1-C_p\delta_n^{p-p'})2^{-\frac{\kappa}{2}n_0}-\frac{1}{2}2^{-\frac{2\kappa}{3}n_0}\geq (1-2^{-\frac{1}{2}\eta(p-p')n_0}-2^{-\frac{\kappa}{6}n_0})2^{-\frac{\kappa}{2}n_0}
\end{equation*}
by~\eqref{eq:FinalIteration}, if $n_0$ is large enough so that we absorb the constant $C_p$ inside the exponential factor for simplicity. The product
\begin{equation*}
    \prod_{n=n_0}^{\infty}(1-2^{-\frac{1}{2}\eta(p-p')n}-2^{-\frac{\kappa}{6}n})
\end{equation*}
converges to some positive limit, so that by choosing $n_0$ large enough we can suppose that it is always greater than $\frac{1}{2}$. Then by induction on~\eqref{eq:FinalIteration} as above, we claim that for any $n\geq n_0$, we have
\begin{equation*}
    \mathbb{E}[\mu_{\text{H}}(L_{r,n+1})^{p}]\geq 2^{-\frac{\kappa}{2}n_0}\prod_{k=n_0}^{n}(1-2^{-\frac{1}{2}\eta(p-p')k}-2^{-\frac{\kappa}{6}k}).
\end{equation*}
Indeed, we have just seen that this is true for $n=n_0$. If this is true for some $n>n_0$, then at rank $n+1$ we have
\begin{equation*}
    u_{n+1}\mathbb{E}[\mu_{\text{H}}(L_{r,n+1})^{p}]\geq (1-2^{-\frac{1}{2}\eta(p-p')(n+1)})\cdot 2^{-\frac{\kappa}{2}n_0}\prod_{k=n_0}^{n}(1-2^{-\frac{1}{2}\eta(p-p')k}-2^{-\frac{\kappa}{6}k})
\end{equation*}
and
\begin{equation*}
    \frac{1}{2}2^{-\frac{2\kappa}{3}(n+1)}=\frac{1}{2}2^{-\frac{\kappa}{6}(n+1)}\cdot 2^{-\frac{\kappa}{2}(n+1)}\leq 2^{-\frac{\kappa}{6}(n+1)}\cdot 2^{-\frac{\kappa}{2}n_0}\prod_{k=n_0}^{n}(1-2^{-\frac{1}{2}\eta(p-p')k}-2^{-\frac{\kappa}{6}k})
\end{equation*}
Together with~\eqref{eq:FinalIteration}, this yields the above inequality for rank $n+1$.

In particular, for some large enough $n_0$ and any $n\geq n_0$,
\begin{equation*}
    \mathbb{E}[\mu_{\text{H}}(L_{r,n+1})^{p}]\geq 2^{-\frac{\kappa}{2}n_0}\prod_{k=n_0}^{n}(1-2^{-\frac{1}{2}\eta(p-p')k}-2^{-\frac{\kappa}{6}k})\geq \frac{1}{2}2^{-\frac{\kappa}{2}n_0}.
\end{equation*}
This shows that $\limsup_{n\to\infty}\mathbb{E}[\mu_{\text{H}}(L_{r,n})^{p}]>0$, so $\mathbb{E}[\mu_{\text{H}}(Q)^{p}]$ cannot be finite. Therefore,
\begin{equation*}
    \mathbb{E}[\mu_{\text{H}}(Q)^{\frac{2}{\gamma^2}}]=\infty.
\end{equation*}
This concludes the proof of Theorem~\ref{th:SeibergSmallNecessary}.
\end{proof}

Finally, for negative moments, it suffices to use the fact that classical Gaussian multiplicative chaos, such as $\mu_{\text{H}}(Q_r^{\text{U}})$, has finite negative moments of all order. Theorem~\ref{th:Main1} for $\gamma\in[\sqrt{2},2)$ follows from Lemma~\ref{lem:ExistenceHalfMoment} and Theorem~\ref{th:SeibergSmallNecessary}.

\section{Proof of the main theorem for large moments}
We continue with the case of small coupling constants $\gamma\in(0,\sqrt{2})$, where the first moment of the hyperbolic Gaussian multiplicative chaos measure is now finite:
\begin{equation}\label{eq:FiniteFirstMoment}
    \forall Q\in\mathcal{Q}^{*}_{\mathbb{R}}, \quad \mathbb{E}\left[\mu_{\text{H}}(Q)\right]=\mathbb{E}\left[\int_{Q}\text{Im}(z)^{-\frac{\gamma^2}{2}}d^2z\right]<\infty.
\end{equation}
The finiteness of the first moment allows us to implement the scheme of~\cite{kahane1976certaines}, originally designed for Gaussian multiplicative cascades models. More precisely, we will prove that, for any $Q\in\mathcal{Q}^{*}_{\mathbb{R}}$,
\begin{enumerate}
    \item $\mathbb{E}\left[\mu_{\text{H}}(Q)^{p}\right]$ explodes for $p\geq\frac{2}{\gamma^2}$;
    \item $\mathbb{E}\left[\mu_{\text{H}}(Q)^{p}\right]$ is finite for $p<\frac{2}{\gamma^2}$, by an extension of the strategy of~\cite{kahane1976certaines}.
\end{enumerate}
Together, these statements prove the extended Seiberg bounds of boundary Liouville conformal field theory of Theorem~\ref{th:Main1} in the regime $\gamma\in(0,\sqrt{2})$. Combined with Theorem~\ref{th:SeibergSmallNecessary}, this completes the proof of~Theorem~\ref{th:Main1}.

\subsection{Explosion of large positive moments}\label{sec:Superadditivity}
We start with the simpler direction, that moments of $\mu_{\text{H}}(Q)$ with $p\geq\frac{2}{\gamma^2}$ explode.
\begin{lemm}[Explosion of large positive moments]
Let $\gamma\in(0,\sqrt{2})$ and $Q\in\mathcal{Q}^{*}_{\mathbb{R}}$. Then for any $p\geq\frac{2}{\gamma^2}$,
\begin{equation*}
    \mathbb{E}[\mu_{\text{H}}(Q)^{p}]=\infty.
\end{equation*}
\end{lemm}
\begin{proof}
    Suppose the contrary, that for some $Q\in\mathcal{Q}^{*}_{\mathbb{R}}$ and $1<\frac{2}{\gamma^2}\leq p<\frac{4}{\gamma^2}$, we have
    \begin{equation*}
        \mathbb{E}[\mu_{\text{H}}(Q)^{p}]<\infty.
    \end{equation*}
    There is no loss in generality by translation invariance, and we can always pass to a smaller Carleson cube if necessary, so we take $Q=Q_r=Q_{[-r,r]}$ in the rest of the proof.

    Decompose $Q_r$ into three parts $Q_r^{\text{L}}, Q_r^{\text{R}}$ and $Q_r^{\text{U}}$ as explained in Section~\ref{sec:Strategy}. By the boundary scaling relation of Lemma~\ref{lem:HyperbolicScaling}, we have
    \begin{equation*}
        \mathbb{E}[\mu_{\text{H}}(Q_r^{\text{L}})^{p}]=\mathbb{E}[\mu_{\text{H}}(Q_r^{\text{R}})^{p}]=2^{-\overline{\zeta}(p)}\mathbb{E}[\mu_{\text{H}}(Q)^{p}],
    \end{equation*}
    while the upper part $Q_r^{\text{U}}$ behaves like a classical Gaussian multiplicative chaos and
    \begin{equation*}
        \mathbb{E}[\mu_{\text{H}}(Q_r^{\text{U}})^{p}]<\infty,
    \end{equation*}
    since we assumed that $p<\frac{4}{\gamma^2}$ (without this assumption, the above moment explodes already). It follows from the superadditivity inequality that
    \begin{equation*}
        \mathbb{E}[\mu_{\text{H}}(Q)^{p}]\geq \mathbb{E}[\mu_{\text{H}}(Q_r^{\text{L}})^{p}]+\mathbb{E}[\mu_{\text{H}}(Q_r^{\text{R}})^{p}]+\mathbb{E}[\mu_{\text{H}}(Q_r^{\text{U}})^{p}]>2^{1-\overline{\zeta}(p)}\mathbb{E}[\mu_{\text{H}}(Q)^{p}],
    \end{equation*}
    which implies that $\mathbb{E}[\mu_{\text{H}}(Q)^{p}]>2^{1-\overline{\zeta}(p)}\mathbb{E}[\mu_{\text{H}}(Q)^{p}]$, so that $1-\overline{\zeta}(p)<0$. But this contradicts Proposition~\ref{prop:TrivialProposition}, since we assumed that $\frac{2}{\gamma^2}\leq p$.
\end{proof}

\subsection{Optimality of the positive moment bound}\label{sec:DecorrelationPlusSokoban}
We now show the other direction. By~\eqref{eq:FiniteFirstMoment} and the fact that the classical Gaussian multiplicative chaos has finite negative moments of all order, we restrict to the moments with $p>1$.
\begin{theo}[The extended Seiberg bound for small coupling constants]\label{th:SeibergLargeCoupling}
Let $\gamma\in(0,\sqrt{2})$ and $Q\in\mathcal{Q}^{*}_{\mathbb{R}}$. Then for any $1<p<\frac{2}{\gamma^2}$,
\begin{equation*}
    \mathbb{E}[\mu_{\text{H}}(Q)^{p}]<\infty.
\end{equation*}
\end{theo}

We recall the arguments for the case $\gamma\in(0,\sqrt{2})$ based on an unpublished appendix of~\cite{huang2018liouville}. The general structure of the proof is inspired by that of~\cite{kahane1976certaines}, and is supplemented with more advanced techniques involving Gaussian decorrelation inequalities together with a combinatorial manipulation, for which we refer to as the Sokoban lemma.

\medskip

We already know by~\eqref{eq:FiniteFirstMoment} that the first moment of $\mu_{\text{H}}(Q)$ is finite, so we can proceed by induction on integers $k\geq 1$ such that $k<p\leq k+1$. We work with $Q=Q_r=Q_{[-r,r]}$, and several intuitions are worth being spotlighted at this point:
\begin{itemize}
    \item By Minkowski's inequality, for $1\leq p<\frac{2}{\gamma^2}$,
    \begin{equation}\label{eq:Minkowski}
    \begin{split}
        \left(\mathbb{E}[\mu_{\text{H}}(Q_r)^{p}]\right)^{1/p}&\leq \left(\mathbb{E}[\mu_{\text{H}}(Q^{\text{L}}_r+Q^{\text{R}}_r)^{p}]\right)^{1/p}+\mathbb{E}[\mu_{\text{H}}(Q^{\text{U}}_r)^{p}]^{1/p}\\
        &\leq \left(\mathbb{E}[\mu_{\text{H}}(Q^{\text{L}}_r+Q^{\text{R}}_r)^{p}]\right)^{1/p}+C_{r,p},
    \end{split}
    \end{equation}
    where $C_{r,p}$ is some constant depending on $r$ and $p$, it is finite since $p<\frac{4}{\gamma^2}$, the latter being the moment bound for the classical Gaussian multiplicative chaos $\mu_{\text{H}}(Q^{\text{U}}_r)$.
    \item If $p=k+1$ were an integer, then by the binomial formula, $\mathbb{E}[\left(\mu_{\text{H}}(Q^{\text{L}}_r)+\mu_{\text{H}}(Q^{\text{R}}_r)\right)^{p}]$ when $p$ is close to $\frac{2}{\gamma^2}$, shoud behave like $\mathbb{E}[\mu_{\text{H}}(Q^{\text{L}}_r)^{p}]+\mathbb{E}[\mu_{\text{H}}(Q^{\text{R}}_r)^{p}]$, up to some cross terms of the form $\mathbb{E}[\mu_{\text{H}}(Q^{\text{L}}_r)^{k}\mu_{\text{H}}(Q^{\text{R}}_r)^{p-k}]$. A more general form of this observation for non-integer $p$ is recalled in Appendix~\ref{sec:Elementary}.
    \item The cross terms of type $\mathbb{E}[\mu_{\text{H}}(Q^{\text{L}}_r)^{k}\mu_{\text{H}}(Q^{\text{R}}_r)^{p-k}]$, when we suppose independence of $\mu_{\text{H}}(Q^{\text{L}}_r)$ and $\mu_{\text{H}}(Q^{\text{R}}_r)$, can be factorized. Each factor is then bounded since the exponents $k$ or $p-k$ are away from the critical moment bound $\frac{2}{\gamma^2}$ by the induction assumption. In reality some dependence exists, and we will use Gaussian decorrelation techniques to give a precise form of this type of bounds.
    \item Finally, by the boundary scaling relation of Lemma~\ref{lem:HyperbolicScaling}, each term $\mathbb{E}[\mu_{\text{H}}(Q^{\text{L}}_r)^{p}]$ is comparable to the original term $\mathbb{E}[\mu_{\text{H}}(Q_r)^{p}]$. We are left to check the sign of the exponent as we have done in the previous proofs.
\end{itemize}

The proof of Theorem~\ref{th:SeibergLargeCoupling} is built on rigorous versions of the above observations, especially on the item where Gaussian decorrelation should be treated carefully. One extra difficulty is due to the fact that $\mathbb{E}[\mu_{\text{H}}(Q_r)^{p}]$ might be infinite to begin with, so we need to go through the regularization procedure for the Gaussian multiplicative chaos measure from the very start. For our tranquillity in reading the rather lengthy proof, it might be useful to point out that the regularization does not reflect the idea of the proof and is only a technical requirement for mathematical rigor. For better readability, some technical details in the following proof are postponed to Section~\ref{sec:GaussianEstimates}.

\begin{proof}[Proof of Theorem~\ref{th:SeibergLargeCoupling}]
The proof is divide into several steps.

\medskip

$\bullet$ \textbf{Step 1: a refined Whitney decomposition.} Without loss of generality, consider $Q=Q_r=Q_{[-r,r]}$ and its Whitney-type decomposition as in Section~\ref{sec:Strategy}. Due to the notationally heavy indices that will appear in the proof, we drop the $r$ index whenever we can in this proof. In view of applying Gaussian decorrelation inequalities, we slightly modify the lower parts of the decomposition. Intuitively, $Q^{\text{L}}$ and $Q^{\text{R}}$ should behave like independent random variables up to a certain multiplicative constant, except for the region near the vertical line where meet. We thus single out this region and treat it separately.

Take a large integer $N\geq 1$ that we fix later, and let $\delta=\frac{1}{N}$ be a small cut-off parameter. Denote by
\begin{equation}
    Q^{\text{L},\delta}\coloneqq [-\delta,0]\times[0,r],\quad Q^{\text{R},\delta}\coloneqq [0,\delta]\times[0,r],
\end{equation}
the $\delta$-slices near the boundary between $Q^{\text{L}}$ and $Q^{\text{R}}$. We also denote, for $1\leq j\leq N$,
\begin{equation*}
    Q^{\text{L},j}\coloneqq [-j\delta,-(j-1)\delta]\times [0,r],\quad Q^{\text{R}, j}\coloneqq [(j-1)\delta,j\delta]\times [0,r],
\end{equation*}
their horizontal translations. See Figure~\ref{fig:RefinedRectangles} for an illustration.

\begin{figure}[h]
\centering
\includegraphics[height=15em]{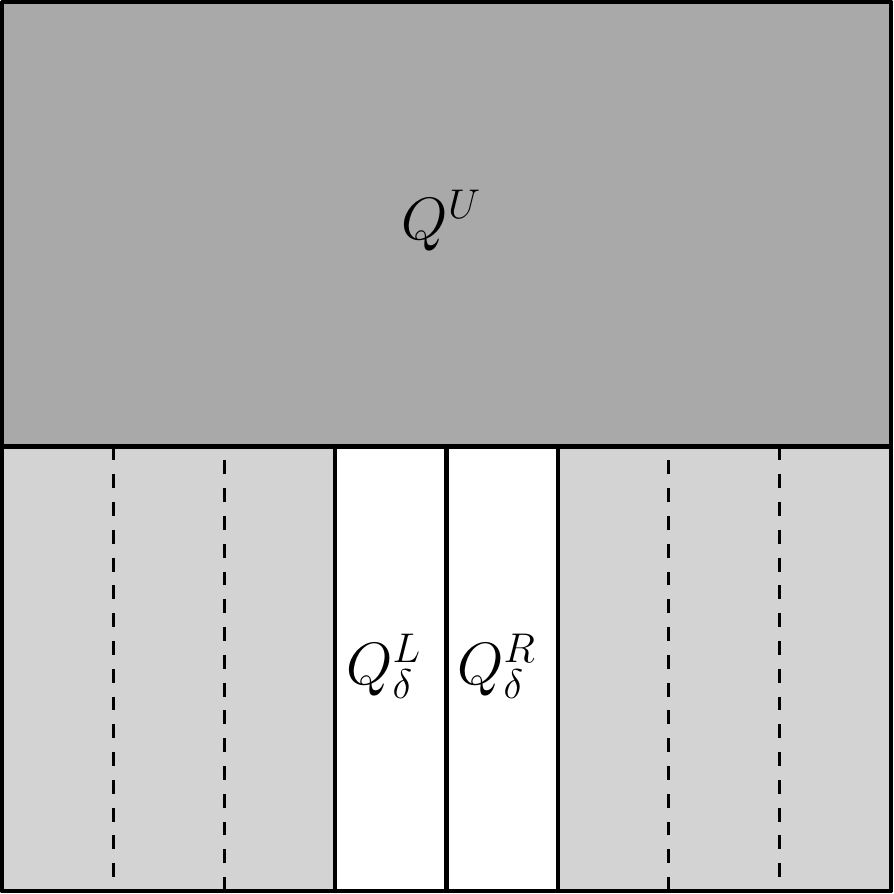}
\caption{Refined Whitney decomposition with slices in the lower parts.}
\label{fig:RefinedRectangles}
\end{figure}

It follows that, modulo some lines where the above-defined rectangles meet (on which the Gaussian multiplicative chaos mass is almost surely null), we have the following partition:
\begin{equation*}
    Q^{\text{L}}=\coprod_{j=1}^{N}Q^{\text{L},j}\eqqcolon Q^{\text{L},\delta}\coprod Q^{\text{L},r-\delta},\quad Q^{\text{R}}=\coprod_{j=1}^{N}Q^{\text{R},j}\eqqcolon Q^{\text{R},\delta}\coprod Q^{\text{R},r-\delta},
\end{equation*}
where we denoted by $Q^{\text{L},r-\delta}$ the complement of $Q^{\text{L},\delta}$ in $Q^{\text{L}}$ (similarly for $Q^{\text{R},r-\delta}$).

\medskip

$\bullet$ \textbf{Step 2: induction scheme and cross-term decomposition.}
The proof goes by induction: we show that if $k\geq 1$, $p\in(k,k+1]$ and $p<\frac{2}{\gamma^2}$, then $\mathbb{E}[\mu_{\text{H}}(Q)^{k}]<\infty$ implies that $\mathbb{E}[\mu_{\text{H}}(Q)^{p}]<\infty$. The initialization $k=1$ is already known by~\eqref{eq:FiniteFirstMoment}.

As is alluded prior to this proof, we should do everything in this proof with regularized Gaussian multiplicative chaos measures. We denote by $\mu_{\text{H},n}$ the $2^{-n}$-regularized Gaussian multiplicative chaos measure as recalled in Lemma~\ref{lem:HyperbolicScaling}. We also denote by $Q_n$ (or $Q^{\text{L}}_n, Q^{\text{R}}_n$ etc.) as $Q\cap\{\text{Im}(y)>2^{-n}\mathrm{d}\}$ for some fixed constant $\mathrm{d}>0$, so that we stay slightly away from the boundary $\mathbb{R}=\partial\mathbb{H}$. The choice of $\mathrm{d}>0$ plays no essential role in the following, to fix ideas one can use $\mathrm{d}=r/2$ for example. These regularizations are only needed for properly initiating the induction scheme, and dropping them (i.e. the index $n$) will not affect the key ideas in the estimates.

The key to the induction is to use the elementary inequality of Proposition~\ref{lem:Muirhead}: in particular,
\begin{equation*}
\begin{split}
    \mathbb{E}[\mu_{\text{H},n}(Q^{\text{L}}_n+Q^{\text{R}}_n)^{p}]\leq 2\mathbb{E}[\mu_{\text{H},n}(Q^{\text{L}}_n)^{p}]+C_p\mathbb{E}[\mu_{\text{H},n}(Q^{\text{L}}_n)^{p}\mu_{\text{H},n}(Q^{\text{R}}_n)^{p-k}]
\end{split}
\end{equation*}
where we used the symmetry between $Q^{\text{L}}$ and $Q^{\text{R}}$ and the constant $C_p$ depends only on $p$.

By the same argument as the boundary scaling relation of Lemma~\ref{lem:HyperbolicScaling}, we rewrite this as
\begin{equation*}
    \mathbb{E}[\mu_{\text{H},n}(Q^{\text{L}}_n+Q^{\text{R}}_n)^{p}]\leq 2^{1-\overline{\zeta}(p)}\mathbb{E}[\mu_{\text{H},n-1}(Q_{n-1})^{p}]+C_p\mathbb{E}[\mu_{\text{H},n}(Q^{\text{L}}_n)^{p}\mu_{\text{H},n}(Q^{\text{R}}_n)^{p-k}].
\end{equation*}
Together with Minkowski's inequality as in~\eqref{eq:Minkowski} and the subadditivity inequality applied to $x\mapsto x^{1/p}$ (now that $p>1$), we arrive at
\begin{equation}\label{eq:Upshot0}
\begin{split}
    &\left(\mathbb{E}[\mu_{\text{H},n}(Q_n)^{p}]\right)^{1/p}\\
    \leq{}&\left(2^{1-\overline{\zeta}(p)}\mathbb{E}[\mu_{\text{H},n-1}(Q_{n-1})^{p}]\right)^{1/p}+\left(C_p\mathbb{E}[\mu_{\text{H},n}(Q^{\text{L}}_n)^{p}\mu_{\text{H},n}(Q^{\text{R}}_n)^{p-k}]\right)^{1/p}+C_{r,p}.
\end{split}
\end{equation}
Notice that by Proposition~\ref{prop:TrivialProposition}, the factor $2^{1-\overline{\zeta}(p)}<1$ since $1<p<\frac{2}{\gamma^2}$. It remains to see that the contribution from the cross term, i.e. middle term in the last equation above, should be small: this is the goal of the next step.

\medskip

$\bullet$ \textbf{Step 3: decorrelation and Sokoban estimates.}
Always under the assumption that for some integer $k$, $1\leq k<p\leq k+1$ and $p<\frac{2}{\gamma^2}$, we now control the cross term
\begin{equation*}
    \mathbb{E}[\mu_{\text{H},n}(Q^{\text{L}}_n)^{k}\mu_{\text{H},n}(Q^{\text{R}}_n)^{p-k}].
\end{equation*}
The idea is to apply the subadditivity inequality to the term $\mu_{\text{H}}(Q^{\text{R}})^{p-k}$, given that $p-k\leq 1$ by definition of $k$. Since 
\begin{equation*}
    Q^{\text{R}}=Q^{\text{R},\delta}\coprod Q^{\text{R},r-\delta},
\end{equation*}
where the notations for this partition are introduced in Step~1, we have
\begin{equation*}
    \mathbb{E}[\mu_{\text{H},n}(Q^{\text{L}}_n)^{k}\mu_{\text{H},n}(Q^{\text{R}}_n)^{p-k}]\leq \mathbb{E}[\mu_{\text{H},n}(Q^{\text{L}}_n)^{k}\mu_{\text{H},n}(Q^{\text{R},r-\delta}_n)^{p-k}]+\mathbb{E}[\mu_{\text{H},n}(Q^{\text{L}}_n)^{k}\mu_{\text{H},n}(Q^{\text{R},\delta}_n)^{p-k}].
\end{equation*}

We have two very different kinds of term above. The first term involves two regions $Q^{\text{L}}$ and $Q^{\text{R},\delta}$ that are separated by some positive cut-off distance $\delta>0$ (see Figure~\ref{fig:RefinedRectangles}). Therefore, we can use Gaussian decorrelation techniques directly and expect a factorization, modulo some multiplicative constant depending on $\delta$. The upshot is
\begin{equation}\label{eq:Upshot1}
\begin{split}
    \mathbb{E}[\mu_{\text{H},n}(Q^{\text{L}})^{k}\mu_{\text{H},n}(Q^{\text{R},r-\delta}_n)^{p-k}]&\leq \delta^{-2k(p-k)\gamma^2}\mathbb{E}[\mu_{\text{H},n}(Q^{\text{L}}_n)^{k}]\mathbb{E}[\mu_{\text{H},n}(Q^{\text{R},r-\delta}_n)^{p-k}]\\
    &\leq C_{p}\delta^{-2k(p-k)\gamma^2},
\end{split}
\end{equation}
where we used the induction hypothesis to bound the factorized expectations. A detailed proof of~\eqref{eq:Upshot1} is presented in Section~\ref{sec:TechnicalEstimates}.

The second term involves two regions $Q^{\text{L}}$ and $Q^{\text{R},\delta}$ that are adjacent (see Figure~\ref{fig:RefinedRectangles}). Therefore, no factorization as in the case of the first term can be expected. Instead, we claim that by choosing $\delta$ small enough, we can make it as small as possible in front of $\mathbb{E}[\mu_{\text{H}}(Q^{\text{L}})^{p}]$. The idea, very informally, is that in average, product of adjacent boxes cannot be much larger than the product of the same box. To do this, we rely on quite heavy manipulations of Gaussian decorrelation inequalities to move the rectangle $Q^{\text{R},\delta}$ around each $\delta$-slices $Q^{\text{L},j}$ (see Step~1 for the definition of the latter). We refer to the moving operations as the Sokoban lemma. The upshot is
\begin{equation}\label{eq:Upshot2}
    \mathbb{E}[\mu_{\text{H},n}(Q^{\text{L}}_n)^{k}\mu_{\text{H},n}(Q^{\text{R},\delta}_n)^{p-k}]\leq C_p\delta^{p-k}\mathbb{E}[\mu_{\text{H},n}(Q^{\text{L}}_n)^{p}].
\end{equation}
The detailed proof of~\eqref{eq:Upshot2} is presented in Section~\ref{sec:SokobanEstimates}.

\medskip

$\bullet$ \textbf{Step 4: putting things together and conclusion.}
It remains to choose the correct cut-off $\delta>0$ to conclude the proof. By combining~\eqref{eq:Upshot0}, \eqref{eq:Upshot1} and \eqref{eq:Upshot2}, we have
\begin{equation*}
\begin{split}
    &\left(\mathbb{E}[\mu_{\text{H},n}(Q_n)^{p}]\right)^{1/p}\\
    \leq{}&\left(2^{1-\overline{\zeta}(p)}\mathbb{E}[\mu_{\text{H},n-1}(Q_{n-1})^{p}]\right)^{1/p}+(C_p\delta^{p-k}\mathbb{E}[\mu_{\text{H},n}(Q^{\text{L}}_n)^{p}])^{1/p}+(C_{p}\delta^{-2k(p-k)\gamma^2})^{1/p}+C_{p}
\end{split}
\end{equation*}
where subadditivity inequality is used. We now choose $\delta>0$ such that
\begin{equation*}
    \frac{2^{(1-\overline{\zeta}(p))\frac{1}{p}}}{1-C_p^{1/p}\delta^{\frac{p-k}{p}}}\eqqcolon\alpha_p<1.
\end{equation*}
This is possible by Proposition~\ref{prop:TrivialProposition} and the assumption that $1\leq p<\frac{2}{\gamma^2}$. With this choice of $\delta$, we have (where $C_{p,\delta}$ changes from line to line)
\begin{equation*}
    \left(\mathbb{E}[\mu_{\text{H},n}(Q_n)^{p}]\right)^{1/p}\leq \alpha_p\left(\mathbb{E}[\mu_{\text{H},n-1}(Q_{n-1})^{p}]\right)^{1/p}+C_{p,\delta}.
\end{equation*}
It is plain to check that by iteration, this implies
\begin{equation*}
    \limsup_{n}\mathbb{E}[\mu_{\text{H},n}(Q_n)^{p}]<\infty.
\end{equation*}
This finishes our induction scheme. It remains to observe that the above equation implies that for all $1\leq p<\frac{2}{\gamma^2}$,
\begin{equation*}
    \mathbb{E}[\mu_{\text{H}}(Q)^{p}]<\infty
\end{equation*}
This completes the proof.
\end{proof}

\section{Gaussian decorrelation techniques and the Sokoban lemma}\label{sec:GaussianEstimates}
We gather in this section some techniques related to Gaussian decorrelation inequalities.

\subsection{Some technical estimates}\label{sec:TechnicalEstimates}
The goal here is to prove Equation~\eqref{eq:Upshot1} using the Gaussian decorrelation inequality of Lemma~\ref{lem:Slepian}. For simplicity we drop the index $n$ below since the proof works independently of this regularization. Let us recast~\eqref{eq:Upshot1} in the form of a proposition (notice that we relax the requirement that $k$ is integer, since we need this in the general case later).
\begin{prop}
Let $k>0$ and $p\in[k,k+1]$. Consider $Q^{\text{L}}$ and $Q^{\text{R},r-\delta}$, two boundary boxes separated by some positive distance $\delta>0$. Then we have the following estimate on their cross moment:
\begin{equation*}
    \mathbb{E}[\mu_{\text{H}}(Q^{\text{L}})^{k}\mu_{\text{H}}(Q^{\text{R},r-\delta})^{p-k}]\leq \delta^{-2k(p-k)\gamma^2}\mathbb{E}[\mu_{\text{H}}(Q^{\text{L}})^{k}]\mathbb{E}[\mu_{\text{H}}(Q^{\text{R},r-\delta})^{p-k}].
\end{equation*}
\end{prop}

\begin{proof}
Make two independent copies of the fields $X$ defined on $Q^{\text{L}}$ and $Q^{\text{R},r-\delta}$, and denote them by $X^{\text{L}}$ and $X^{\text{R},r-\delta}$. Consider also some independent standard Gaussian random variables $N, N^{\text{L}}$ and $N^{\text{R},r-\delta}$. Consider the following Gaussian fields defined on $T=Q^{\text{L}}\cup Q^{\text{R},r-\delta}$:
\begin{equation*}
    X^{\text{L}}\mathbf{1}_{Q^{\text{L}}}+X^{\text{R},r-\delta}\mathbf{1}_{Q^{\text{R},r-\delta}}+\sqrt{-2\ln\delta}N \quad\text{and}\quad X+\sqrt{-2\ln\delta}N^{\text{L}}\mathbf{1}_{Q^{\text{L}}}+\sqrt{-2\ln\delta}N^{\text{R},r-\delta}\mathbf{1}_{Q^{\text{R},r-\delta}}.
\end{equation*}
The first field dominates the second field in covariance on $Q^{\text{L}}\times Q^{\text{R},r-\delta}\subset T^2$ (see Appendix~\ref{app:Decorrelation} and Lemma~\ref{lem:Slepian} for the definition), and their covariances are equal elsewhere on $T^2$. Writing a discretized version of the product $\mu_{\text{H}}(Q^{\text{L}})^{k}\mu_{\text{H}}(Q^{\text{R},r-\delta})^{p-k}$ as functional of the Gaussian field $X$, it is also standard to verify the partial derivative condition of Lemma~\ref{lem:Slepian}. Therefore, Lemma~\ref{lem:Slepian} yields
\begin{equation*}
\begin{split}
    &\mathbb{E}[e^{\gamma k \sqrt{-2\ln\delta}N^{\text{L}}}]\mathbb{E}[e^{\gamma (p-k) \sqrt{-2\ln\delta}N^{\text{R},r-\delta}}]\mathbb{E}[\mu_{\text{H}}(Q^{\text{L}})^{k}\mu_{\text{H}}(Q^{\text{R},r-\delta})^{p-k}]\\
    \leq{}&\mathbb{E}[e^{\gamma k \sqrt{-2\ln\delta}N} e^{\gamma (p-k) \sqrt{-2\ln\delta}N}]\mathbb{E}[\mu_{\text{H}}(Q^{\text{L}})^{k}]\mathbb{E}[\mu_{\text{H}}(Q^{\text{R},r-\delta})^{p-k}].
\end{split}
\end{equation*}
Rearranging we get
\begin{equation*}
    \mathbb{E}[\mu_{\text{H}}(Q^{\text{L}})^{k}\mu_{\text{H}}(Q^{\text{R},r-\delta})^{p-k}]\leq e^{2\gamma^2k(p-k)(-\ln\delta)}\mathbb{E}[\mu_{\text{H}}(Q^{\text{L}})^{k}]\mathbb{E}[\mu_{\text{H}}(Q^{\text{R},r-\delta})^{p-k}].
\end{equation*}
This yields the proposition.
\end{proof}

\subsection{A Sokoban lemma}\label{sec:SokobanEstimates}
The goal here is to prove Equation~\eqref{eq:Upshot2}. Again, we drop the regularization index $n$ since it is of no importance here. We recast~\eqref{eq:Upshot2} in the form of a lemma (again, notice that we relax the requirement that $k$ is integer).
\begin{lemm}[The Sokoban lemma]
Let $k>0$ and $p\in[k,k+1]$. Consider $Q^{\text{L}}$ and $Q^{\text{R},\delta}$, two adjacent boundary boxes. Then we have the following estimate on their cross moment:
\begin{equation*}
    \mathbb{E}[\mu_{\text{H}}(Q^{\text{L}})^{k}\mu_{\text{H},n}(Q^{\text{R},\delta})^{p-k}]\leq C_p\delta^{p-k}\mathbb{E}[\mu_{\text{H}}(Q^{\text{L}})^{p}].
\end{equation*}
\end{lemm}
We prove this in the case of $Q^{\text{L}}$, since the proof for the other case is almost identical.

\begin{proof}
By assumption $p-k\leq 1$, we can use Jensen's inequality to write (recall that $\delta=\frac{1}{N}$)
\begin{equation*}
    \mathbb{E}[\mu_{\text{H}}(Q^{\text{L}})^{p}]\geq \delta^{1-(p-k)}\sum_{j=1}^{N}\mathbb{E}[\mu_{\text{H}}(Q^{\text{L}})^{k}\mu_{\text{H}}(Q^{\text{L},j})^{p-k}],
\end{equation*}
where $Q^{\text{L},j}$ is defined in Step~1 of the proof of Theorem~\ref{th:SeibergLargeCoupling} as translations of $Q^{\text{L},\delta}$. Therefore, it suffices to show
\begin{equation}\label{eq:Sokoban}
    \mathbb{E}[\mu_{\text{H}}(Q^{\text{L}})^{k}\mu_{\text{H}}(Q^{\text{R},\delta})^{p-k}]\leq C_{p}\mathbb{E}[\mu_{\text{H}}(Q^{\text{L}})^{k}\mu_{\text{H}}(Q^{\text{L},j})^{p-k}]
\end{equation}
with a constant $C_p$ independent of $j\in\{1,\dots,N\}$.

This translates to the following heuristic statement, ``moving the box $Q^{\text{R},\delta}$ inside of $Q^{\text{L}}$ will not increase the order of the cross term estimate''. To prove this statement, we need several preliminaries around moving boxes and cross term estimates. These estimates can all be shown in a similar manner as in the proof of Equation~\eqref{eq:Upshot1} of Section~\ref{sec:TechnicalEstimates}, so we only give below the Gaussian fields on which to apply Lemma~\ref{lem:Slepian}.

\medskip

The first observation is the following statement, ``moving the box $Q^{\text{R},\delta}$ away from $Q^{\text{L}}$ will only decrease the cross term estimate''. More precisely, for all $j\in\{1,\dots,N\}$,
\begin{equation}\label{eq:MovingAway}
    \mathbb{E}[\mu_{\text{H}}(Q^{\text{L}})^{k}\mu_{\text{H}}(Q^{\text{R},\delta})^{p-k}]\geq \mathbb{E}[\mu_{\text{H}}(Q^{\text{L}})^{k}\mu_{\text{H}}(Q^{\text{R},j})^{p-k}].
\end{equation}
This can be shown by applying Lemma~\ref{lem:Slepian} to the following Gaussian fields with $T=Q^{\text{L}}\cup Q^{\text{R},\delta}$ and $A=(Q^{\text{L}}\times Q^{\text{R},\delta})\cup (Q^{\text{R},\delta}\times Q^{\text{L}})$, $B=\emptyset$:
\begin{equation*}
    X^{\text{L}}\mathbf{1}_{Q^{\text{L}}}+X^{\text{R},\delta}\mathbf{1}_{Q^{\text{R},\delta}} \quad\text{and}\quad X^{\text{L}}\mathbf{1}_{Q^{\text{L}}}+X^{\text{R},j}(\cdot+(j-1)(\delta,0))\mathbf{1}_{Q^{\text{R},\delta}},
\end{equation*}
where $X^{\text{R},j}(\cdot+(j-1)(\delta,0))\mathbf{1}_{Q^{\text{R},\delta}}$ denotes the horizontal translation of the box. See Figure~\ref{fig:MovingAway} for a sketch.

\begin{figure}[h]
\centering
\includegraphics[height=8em]{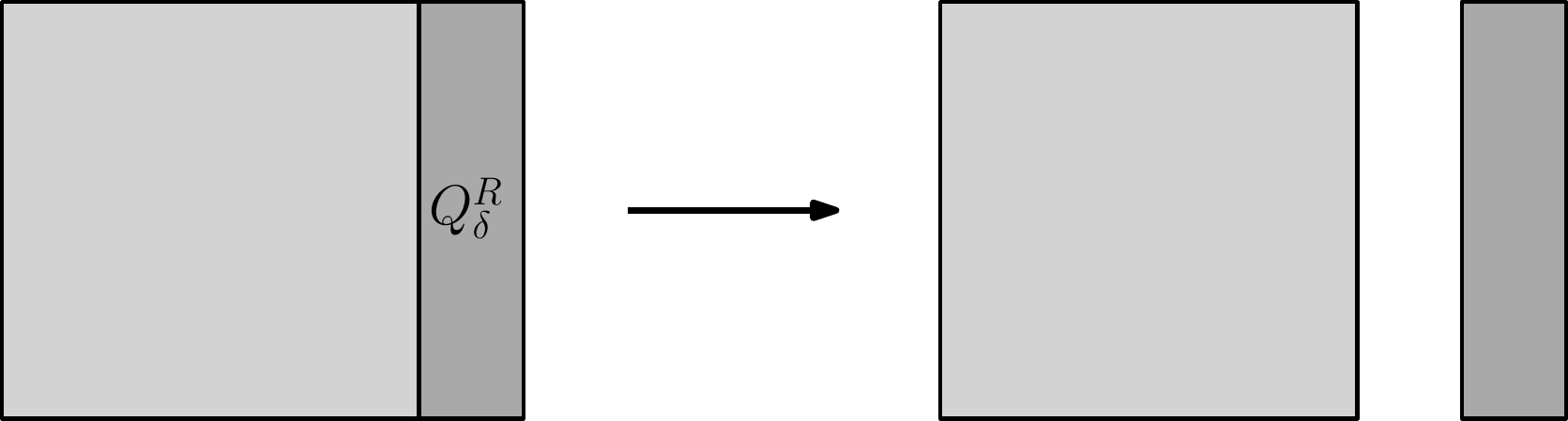}
\caption{Moving a box away from another box.}
\label{fig:MovingAway}
\end{figure}

\medskip

The second observation is the following statement, ``moving the box $Q^{\text{R},\delta}$ to $Q^{\text{L},\delta}$ will only increase the cross term estimate''. Another way of putting it, we reflect the box $Q^{\text{R},\delta}$ along the boundary between $Q^{\text{L}}$ and $Q^{\text{R}}$. More precisely,
\begin{equation}\label{eq:MovingReflect}
    \mathbb{E}[\mu_{\text{H}}(Q^{\text{L}})^{k}\mu_{\text{H}}(Q^{\text{R},\delta})^{p-k}]\leq \mathbb{E}[\mu_{\text{H}}(Q^{\text{L}})^{k}\mu_{\text{H}}(Q^{\text{L},\delta})^{p-k}].
\end{equation}
This can be shown by applying Lemma~\ref{lem:Slepian} to the following Gaussian fields with $T=Q^{\text{L}}\cup Q^{\text{R},\delta}$ and $A=(Q^{\text{L}}\times Q^{\text{R},\delta})\cup (Q^{\text{R},\delta}\times Q^{\text{L}})$, $B=\emptyset$:
\begin{equation*}
    X^{\text{L}}\mathbf{1}_{Q^{\text{L}}}+X^{\text{R},\delta}\mathbf{1}_{Q^{\text{R},\delta}} \quad\text{and}\quad X^{\text{L}}\mathbf{1}_{Q^{\text{L}}}+X^{\text{L},\delta}(-1 \times \overline{\cdot})\mathbf{1}_{Q^{\text{R},\delta}},
\end{equation*}
where $X^{\text{L},\delta}(-1 \times \overline{\cdot})\mathbf{1}_{Q^{\text{R},\delta}}$ denotes the reflection of the box with respect to the axis $x=0$. See Figure~\ref{fig:Reflection} for a sketch.

\begin{figure}[h]
\centering
\includegraphics[height=8em]{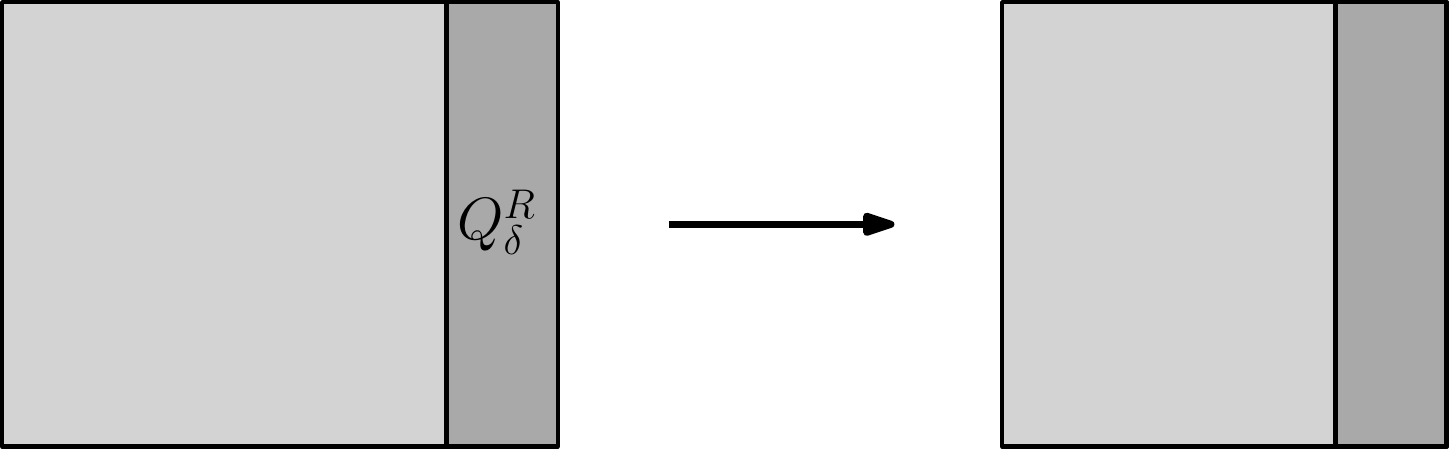}
\caption{Reflecting a box with respect to the common boundary.}
\label{fig:Reflection}
\end{figure}

We are now in a position to show Equation~\eqref{eq:Sokoban} using these box-moving estimates. The case $j=1$ is the reflecting case of the second observation above. For general $j\in\{1,\dots,n\}$, we cut $Q^{\text{L}}$ into two parts,
\begin{equation*}
    Q^{\text{L},\leq j}\coloneqq \bigcup_{i=1}^{j}Q^{\text{L},i}\quad\text{and}\quad Q^{\text{L},>j}\coloneqq \bigcup_{i=j+1}^{N}Q^{\text{L},i}
\end{equation*}
Then by Jensen's inequality, to show Equation~\eqref{eq:Sokoban} it suffices to show that
\begin{equation*}
    \mathbb{E}[\mu_{\text{H}}(Q^{\text{L},>j})^{k}\mu_{\text{H}}(Q^{\text{R},\delta})^{p-k}]\leq C_{p}\mathbb{E}[\mu_{\text{H}}(Q^{\text{L},>j})^{k}\mu_{\text{H}}(Q^{\text{L},j})^{p-k}]
\end{equation*}
and
\begin{equation*}
    \mathbb{E}[\mu_{\text{H}}(Q^{\text{L},\leq j})^{k}\mu_{\text{H}}(Q^{\text{R},\delta})^{p-k}]\leq C_{p}\mathbb{E}[\mu_{\text{H}}(Q^{\text{L},\leq j})^{k}\mu_{\text{H}}(Q^{\text{L},j})^{p-k}].
\end{equation*}
See Figure~\ref{fig:CuttingBoxes} for a sketch.

\begin{figure}[h]
\centering
\includegraphics[height=8em]{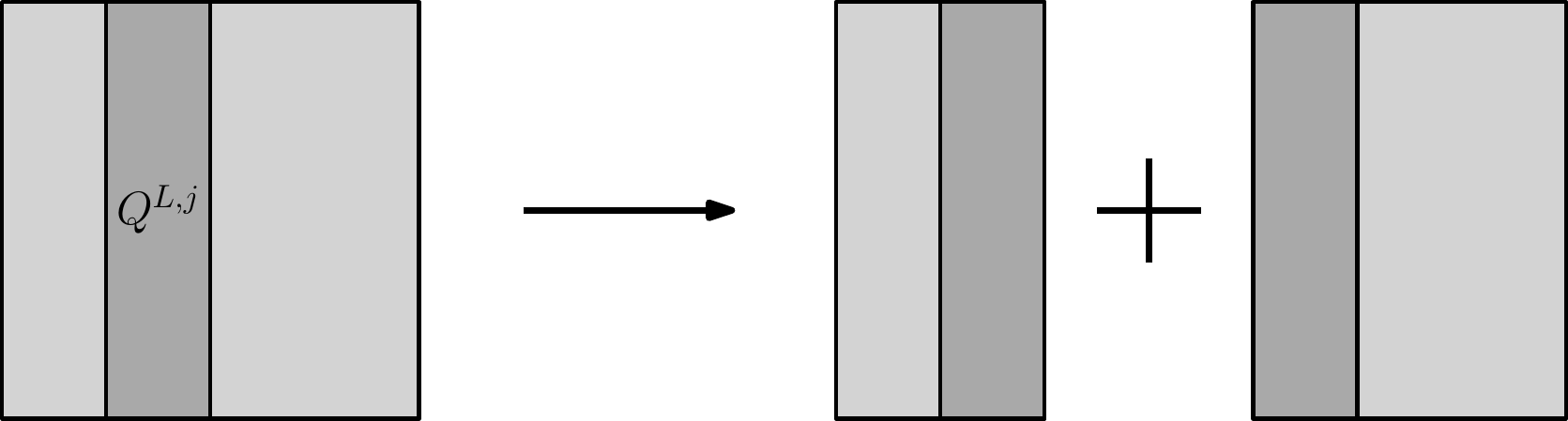}
\caption{Cutting the box $Q^{\text{L}}$ with respect to $Q^{\text{L},j}$.}
\label{fig:CuttingBoxes}
\end{figure}

The first inequality results from the first observation above (the ``moving away'' estimate), and the second inequality results from the second observation above (the ``reflecting'' estimate), both by replacing $Q^{\text{L}}$ in the above argument with suitable boxes $Q^{\text{L},\leq j}$ or $Q^{\text{L},>j}$. This finishes the proof of the lemma. 
\end{proof}

Finally, we record a consequence of this section that is of general interest. The following corollary quantifies the almost ``factorization'' property of the cross-terms involving two adjacent boxes. This estimate is used in the proof of Theorem~\ref{th:SeibergSmallNecessary}.

\begin{coro}\label{coro:AlmostFactorization}
Let $p>0$ and $0\leq q\leq 1$. Consider the two rectangles $L^{\text{L}}\subset Q^{\text{L}}$ and $L^{\text{R}}\subset Q^{\text{R}}$ in Figure~\ref{fig:HorizontalRectangles}, contained in the adjacent boxes $Q^{\text{L}}$ and $Q^{\text{R}}$. Then the following cross moment estimate holds for any $\delta=\frac{1}{N}$ with large enough integer $N$:
\begin{equation*}
\begin{split}
    &\mathbb{E}[\mu_{\text{H}}(L^{\text{L}})^{p}\mu_{\text{H}}(L^{\text{R}})^{q}]\\
    \leq{}&C_{p+q}\delta^{q}\mathbb{E}[\mu_{\text{H}}(L^{L})^{p+q}]+\delta^{-2\gamma^2pq}\mathbb{E}[\mu_{\text{H}}(L^{\text{L}})^{p}]\mathbb{E}[\mu_{\text{H}}(L^{\text{R}})^{q}].
\end{split}
\end{equation*}
\end{coro}
\begin{proof}
Cut the rectangle $L^{\text{R}}$ into two parts: $L^{\text{R},\delta}\coloneqq L^{\text{R}}\cap Q^{\text{R},\delta}$ which is adjacent to $Q^{\text{L}}$, and $L^{\text{R},r-\delta}\coloneqq L^{\text{R}}\cap Q^{\text{R},r-\delta}$ which is $\delta$-away from $Q^{\text{L}}$. Since $0\leq q\leq 1$, via subadditivity it suffices to see that, by the same argument as the Sokoban lemma in Section~\ref{sec:SokobanEstimates},
\begin{equation*}
    \mathbb{E}[\mu_{\text{H}}(L^{\text{L}})^{p}\mu_{\text{H}}(L^{\text{R},\delta})^{q}]\leq C_{p+q}\delta^{q}\mathbb{E}[\mu_{\text{H}}(L^{L})^{p+q}],
\end{equation*}
and also by the same arguemnt in Section~\ref{sec:TechnicalEstimates},
\begin{equation*}
    \mathbb{E}[\mu_{\text{H}}(L^{\text{L}})^{p}\mu_{\text{H}}(L^{\text{R},r-\delta})^{q}]\leq \delta^{-2\gamma^2pq}\mathbb{E}[\mu_{\text{H}}(L^{\text{L}})^{p}]\mathbb{E}[\mu_{\text{H}}(L^{\text{R},r-\delta})^{q}].
\end{equation*}
It remains to use the basic fact that $\mathbb{E}[\mu_{\text{H}}(L^{\text{R},r-\delta})^{q}]\leq \mathbb{E}[\mu_{\text{H}}(L^{\text{R}})^{q}]$.
\end{proof}

\appendix
\section{Some elementary inequalities}\label{sec:Elementary}
We gather some elementary inequalities used in the paper. They serve mainly as converses to the superadditive/subadditive-type inequalities.

\begin{prop}[Converse of the superadditivity inequality]\label{lem:Muirhead}
Let $x,y\geq 0$ be real numbers. Let $k\geq 1$ be a positive integer and $0\leq q\leq 1$. Then
\begin{equation*}
    (x+y)^{k+q}\leq x^{k+q}+y^{k+q}+C_{k}(x^{k}y^{q}+x^{q}y^{k}),
\end{equation*}
where $C_{k}$ is some constant depending only on $k$.
\end{prop}

\begin{prop}[Converse of the subadditivity inequality]\label{lem:ConverseSubadditivity}
Let $x,y\geq 0$ be real numbers. Let $p,q>0$ such that $\frac{1}{2}\leq p+q\leq 1$. Then
\begin{equation*}
    (x+y)^{p+q}\geq x^{p+q}+y^{p+q}-2(x^{p}y^{q}+x^{q}y^{p}).
\end{equation*}
\end{prop}

The proofs of these inequalities are elementary. They can be seen as variants of the following classical lemma:
\begin{prop}[Muirhead's inequality]\label{prop:Muirhead}
Let $0\leq p_1\leq q_1$, $0\leq p_2\leq q_2$ such that $p_1+q_1=p_2+q_2$. Then for any $x,y\geq 0$,
\begin{equation*}
    x^{p_1}y^{q_1}+x^{q_1}y^{p_1}\geq x^{p_2}y^{q_2}+x^{q_2}y^{p_2}
\end{equation*}
if and only if $|p_1-q_1|\geq |p_2-q_2|$.
\end{prop}
\begin{proof}
By symmetry we can suppose $|p_1-q_1|\geq |p_2-q_2|$. By symmetry, we can further suppose $0\leq p_1\leq p_2\leq q_2\leq q_1$. Dividing by $x^{p_1}y^{p_1}\geq 0$, we can suppose $p_1=0$. The inequality is reduced to
\begin{equation*}
    y^{q_1}+x^{q_1}\geq x^{p_2}y^{q_2}+x^{q_2}y^{p_2}.
\end{equation*}
for $0\leq p_2\leq q_2\leq q_1=p_2+q_2$. This follows then from developing the product
\begin{equation*}
    (x^{p_2}-y^{p_2})(x^{q_2}-y^{q_2})\geq 0.
\end{equation*}
Muirhead's inequality then follows.
\end{proof}

Now we can show the other inequalities above.
\begin{proof}[Proof of Proposition~\ref{lem:Muirhead}]
Write $(x+y)^{k+q}=((x+y)^{\frac{k+q}{k+1}})^{k+1}\leq (x^{\frac{k+q}{k+1}}+y^{\frac{k+q}{k+1}})^{k+1}$ by subadditivity, since $k+q\leq k+1$. Developing the integer power by binomial expansion, we have
\begin{equation*}
    (x^{\frac{k+q}{k+1}}+y^{\frac{k+q}{k+1}})^{k+1}=x^{k+q}+y^{k+q}+\dots
\end{equation*}
where in the $\dots$ are $2^{k+1}-2$ cross terms of the form $x^{p}y^{k+q-p}$, with $\frac{k+q}{k+1}\leq p\leq\frac{k(k+q)}{k+1}$. Since $q\leq\frac{k+q}{k+1}$, the maximum of the difference of powers $|(k+q-p)-p|$ in the cross terms is less than $|k-q|$. Therefore by Muirhead's inequality of Proposition~\ref{prop:Muirhead}, the sum of the cross terms is less than
\begin{equation*}
    (2^{k}-1)(x^{k}y^{q}+x^{q}y^{k}).
\end{equation*}
Therefore Proposition~\ref{lem:Muirhead} follows with $C_{k}=2^{k}-1$.
\end{proof}

\begin{proof}[Proof of Proposition~\ref{lem:ConverseSubadditivity}]
By Proposition~\ref{prop:Muirhead}, $x^{p}y^{q}+x^{q}y^{p}\geq 2x^{\frac{p+q}{2}}y^{\frac{p+q}{2}}$. Therefore, it suffices to show the case with $p=q$. But
\begin{equation*}
\begin{split}
    x^{2p}+y^{2p}&=((x^{2p}+y^{2p})^2)^{\frac{1}{2}}\\
    &=(x^{4p}+y^{4p}+2x^{2p}y^{2p})^{\frac{1}{2}}\\
    &\leq (x^{4p}+y^{4p})^{\frac{1}{2}}+\sqrt{2}x^{p}y^{p}
\end{split}
\end{equation*}
by subadditivity. Furthermore, since $4p\geq 1$, superadditivity yields
\begin{equation*}
    x^{4p}+y^{4p}\leq (x+y)^{4p},
\end{equation*}
so that we get
\begin{equation*}
    x^{2p}+y^{2p}\leq ((x+y)^{4p})^{\frac{1}{2}}+\sqrt{2}x^{p}y^{p}=(x+y)^{2p}+\sqrt{2}x^{p}y^{p}.
\end{equation*}
This finishes the proof (with a better constant $\sqrt{2}$ then $2$, but this is of no importance for our purpose).
\end{proof}

\section{Gaussian decorrelation inequalities}\label{app:Decorrelation}
We gather some Gaussian decorrelation inequalities used in the paper. For simplicity, we record the discrete sum version of these inequalities below, and for most applications in this paper, it suffices to use a Riemann sum approximation to get the integral versions.

Given two \emph{centered} Gaussian vectors $\bm{X}=\{X_i\}_{i\in T}$ and $\bm{Y}=\{Y_i\}_{i\in T}$ where $T$ is some discrete index space, we say that $Y$ dominates $X$ in covariance on $A\subset T^2$ if for any $(i,j)\in A$, $\mathbb{E}[X_iX_j]\leq\mathbb{E}[Y_iY_j]$. We say that a function $f$ on $\mathbb{R}^{n}$ has subgaussian growth if for any $\epsilon\geq 0$, there is some constant $C$ such that $|f(x)|\leq Ce^{\epsilon|x|^2}$ for all $x\in\mathbb{R}^{n}$.

\begin{lemm}[Slepian's lemma]\label{lem:Slepian}
Let $\bm{X}$, $\bm{Y}$ be two centered Gaussian vectors indexed by $T$. Let $A,B$ be disjoint subsets of $T^2$ and suppose that $\bm{Y}$ dominates $\bm{X}$ in covariance on $T^2\setminus B$ and $\bm{X}$ dominates $\bm{Y}$ in covariance on $T^2\setminus A$. Suppose that $F$ is some smooth functional with subgaussian growth at infinity as well as in its first and second derivative. If furthermore $\partial F_{ij}\geq 0$ for $(i,j)\in A$ and $\partial F_{ij}\leq 0$ for $(i,j)\in B$, then
\begin{equation*}
    \mathbb{E}[F(\bm{X})]\leq\mathbb{E}[F(\bm{Y})].
\end{equation*}
\end{lemm}

\begin{lemm}[Kahane's convexity inequality]\label{lem:KahaneConvexity}
Let $\bm{X}$, $\bm{Y}$ be two Gaussian vectors indexed by $T$ such that $\bm{Y}$ dominates $\bm{X}$ in covariance on $T^2$. Then for any \emph{convex} functional $F$ with subgaussian growth at infinity as well as in its first and second derivatives, and for any \emph{positive} weights $\{p_i\}_{i\in T}$,
\begin{equation*}
    \mathbb{E}[F(\sum_{i\in T}p_ie^{X_i-\frac{1}{2}\mathbb{E}[X_i^2]})]\leq \mathbb{E}[F(\sum_{i\in T}p_ie^{Y_i-\frac{1}{2}\mathbb{E}[Y_i^2]})].
\end{equation*}
\end{lemm}

The proofs of these inequalities can be found in e.g.~\cite[Section~3]{zeitouniGaussian}. The following classical corollary is important for studying moments of Gaussian multiplicative chaos.

\begin{coro}\label{coro:ClassicalApplicationKahane}\label{coro:KahaneApplication}
Let $X$, $Y$ be two log-correlated Gaussian fields defined on $\Omega\subset\mathbb{R}^{d}$ such that for any $(z,w)\in \Omega^2$,
\begin{equation*}
    \mathbb{E}[X(z)X(w)]\leq \mathbb{E}[Y(z)Y(w)]+R
\end{equation*}
for some $R>0$. Then for any $\gamma\in(0,\sqrt{2d})$ and $p\in\mathbb{R}$, there is some finite constant $C=C(R,\gamma,p)$ depending only on $R,\gamma$ and $p$ such that for all measurable sets $A\subset\Omega$,
\begin{equation*}
    \mathbb{E}[\mu^{\gamma}_{X}(A)^{p}]\leq C(R,\gamma,p)\mathbb{E}[\mu^{\gamma}_{Y}(A)^{p}],
\end{equation*}
where $\mu_{X}$ and $\mu_{Y}$ are the Gaussian multiplicative chaos measures associated to $X$ and $Y$ with parameter $\gamma$, respectively.
\end{coro}
\begin{proof}
From the relation of covariances we deduce that, if $N$ is an independent standard Gaussian variable, the field $\{Y(z)+\sqrt{R}N\}_{z\in\Omega}$ dominates in covariance the field $\{X(z)\}_{z\in\Omega}$ on $\Omega^2$. Therefore, applying Kahane's convexity inequality and factorizing out the renormalized exponential of $\sqrt{R}N$, we get
\begin{equation*}
    \mathbb{E}[\mu^{\gamma}_{X}(A)^{p}]\leq \mathbb{E}[(e^{\gamma\sqrt{R}N-\frac{\gamma^2}{2}R})^{p}]\mathbb{E}[\mu^{\gamma}_{Y}(A)^{p}]
\end{equation*}
One checks that the multiplicative factor in the last term is finite and depends only on $R,\gamma$ and $p$.
\end{proof}

In Section~\ref{sec:TechnicalEstimates} the ideas of the above proof was extended to the case with Lemma~\ref{lem:Slepian}.

\bibliographystyle{alpha}

\end{document}